\documentclass[a4paper]{amsart}

\usepackage[utf8]{inputenc}
\usepackage[T1]{fontenc}
\usepackage{lmodern}
\usepackage{amssymb}
\usepackage[all]{xy}
\usepackage{nicefrac,mathtools}
\usepackage{microtype}
\usepackage{enumerate}
\usepackage[pdftitle={Inverse semigroup expansions},
  pdfauthor={Alcides Buss and Ruy Exel},
  pdfsubject={Mathematics; ...}]{hyperref}
\usepackage[lite]{amsrefs}
\newcommand*{\MRref}[2]{ \href{http://www.ams.org/mathscinet-getitem?mr=#1}{MR #1}}
\newcommand*{\arxiv}[1]{ \href{http://www.arxiv.org/abs/#1}{arXiv:#1}}

\newtheorem{theorem}{Theorem}[section]
\numberwithin{equation}{section}
\newtheorem{lemma}[theorem]{Lemma}
\newtheorem{proposition}[theorem]{Proposition}
\newtheorem{corollary}[theorem]{Corollary}
\newtheorem{definition}[theorem]{Definition}

\theoremstyle{remark}
\newtheorem{remark}{Remark}

\DeclareMathOperator*{\dom}{dom}% domain of a map
\DeclareMathOperator*{\ran}{ran}% range of a map
\DeclareMathOperator*{\cspn}{\overline{span}}% closed linear span
\DeclareMathOperator*{\Inv}{Inv}% category of inverse semigroups with homomorphisms
\DeclareMathOperator*{\Pinv}{PInv}% category of inverse semigroups with partial homomorphisms
\DeclareMathOperator*{\Hom}{Hom}% Hom set
\DeclareMathOperator*{\PHom}{PHom}% Partial Hom set
\DeclareMathOperator*{\Id}{Id}% Identity

\newcommand*{\Star}{\texorpdfstring{$^*$\nobreakdash-\hspace{0pt}}{*-}}
\newcommand*{\Ad}{\mathrm{Ad}}% Adjoint by a unitary

\newcommand*{\sbe}{\subseteq}
\newcommand*{\red}{\textup{r}} % reduced
\newcommand*{\cstar}{\texorpdfstring{$C^*$\nobreakdash-\hspace{0pt}}{C*-}}
\newcommand*{\cont}{\mathcal C} % continuous functions
\newcommand*{\contz}{\cont_0} % continuous functions vanishing at infinity
\newcommand*{\id}{\textup{id}} % identity map

\newcommand*{\hils}{\mathcal H} % Hilbert space
\newcommand*{\A}{\mathcal{A}} % other notation for Fell bundle 1
\newcommand*{\B}{\mathcal{B}} % other notation for Fell bundle 2
\newcommand*{\D}{\mathcal{D}} % notation for ideal of a twisted partial action
\newcommand*{\F}{\mathcal{F}} % hilbert module
\newcommand*{\M}{\mathcal{M}} % multiplier algebras
\newcommand*{\E}{\mathcal{E}} % idempotent part of a Fell bundle
\newcommand*{\U}{\mathcal{U}} % Unitaries

\newcommand*{\defeq}{\mathrel{\vcentcolon=}}

\newcommand*{\into}{\hookrightarrow}

\newcommand*{\congto}{\xrightarrow\sim}

\newcommand*{\Ls}{\mathcal{L}} % adjointable operators
\newcommand*{\inv}{^{-1}} % inversion
\newcommand*{\G}{G} % inverse semigroup
\renewcommand*{\H}{H} % another inverse semigroup
\newcommand*{\st}{\colon}

\newcommand*{\I}{\mathcal{I}}
\newcommand*{\e}{\varepsilon_}
\newcommand*{\T}{H}

\newcommand*{\Pl}{\mathcal{P}_\ell}

\renewcommand*{\S}[1]{\Pr(#1)} % expansion of an inverse semigroup
\newcommand*{\SG}{\S\G} % expansion of the inverse semigroup \G
\newcommand*{\Sh}[1]{\tilde{#1}}
\newcommand*{\pros}{\textstyle\prod\limits}

\newcommand*{\Sp}{{\bf Pr}(\G)}
\renewcommand*{\Pr}{{\bf Pr}} % Prefix expansion functor
\newcommand*{\Fg}{{\bf Fg}} % Forgetful functor
\newcommand*{\gen}[1]{\langle #1 \rangle} % filter generated by a filter base
\newcommand*{\eq}[1]{\buildrel{#1}\over =}

\begin{document}
\title[Inverse semigroup expansions]{Inverse semigroup expansions \\ and their actions on \cstar{}algebras}

\author{Alcides Buss}
\email{alcides@mtm.ufsc.br}
\author{Ruy Exel}
\email{exel@mtm.ufsc.br}

\address{Departamento de Matemática\\
  Universidade Federal de Santa Catarina\\
  88.040-900 Florianópolis-SC\\
  Brasil}

\begin{abstract}
In this work, we give a presentation of the prefix expansion $\SG$ of an inverse semigroup $\G$ as recently introduced by Lawson, Margolis and Steinberg which is similar to the universal inverse semigroup defined by the
second named author in case $\G$ is a group. The inverse semigroup $\SG$ classifies the partial actions of $G$ on spaces. We extend this result and prove that Fell bundles over $\G$ correspond bijectively to saturated Fell
bundles over $\SG$. In particular, this shows that twisted partial actions of $\G$ (on \cstar{}algebras) correspond to twisted (global) actions of $\SG$. Furthermore, we show that this correspondence preserves
\cstar{}algebras crossed products.
\end{abstract}



\thanks{This research was partially supported by CNPq.}

\maketitle

 \tableofcontents

\section{Introduction}
\label{sec:introduction}

Modifying the construction of the Birget-Rhodes prefix expansion of a
semigroup~\cite{BirgetRhodes:expansion}, Lawson, Margolis and Steinberg~\cite{LawsonMargolisSteinberg:Expansions} introduced a generalized
prefix expansion of an inverse semigroup $\G$ as follows: let $\Sp$ be the
collection of all pairs $(A,t)$, where $t\in\G$, and $A$ is a finite subset of
$\G$ satisfying
  $$
  tt^*,t\in A \quad\mbox{and}\quad ss^*=tt^* \quad\mbox{for all } s\in A.
  $$
  With the operation
  \begin{equation}\label{LMSFormulaForProduct}
  (A,t)(B, s) = (tss^*t^*A\cup tB, ts),
  \end{equation}
  $\Sp$ turns out to be an inverse semigroup, and in the main application
presented in \cite{LawsonMargolisSteinberg:Expansions}, it is proved that $\Sp$ possesses a universal property
with respect to partial actions of $\G$: partial actions (resp. representations/homomorphisms) of $\G$ correspond bijectively to global actions (resp. representations/homomorphisms) of $\SG$.

The above explicit construction of $\Sp$ is suggested via a McAlister-O'Carroll triple \cite{Lawson:InverseSemigroups} constructed from $\G$ in a somewhat involving way.

The first goal of this paper is to obtain an alternate % would it not be "alternative"?
description of $\Sp$ in terms of generators and relations which emphasizes its role as governing partial representations (or homomorphisms) of $\G$.  Our method for doing so is independent of
\cite{LawsonMargolisSteinberg:Expansions} and is very similar to the one adopted in \cite{Exel:PartialActionsGroupsAndInverseSemigroups}, where an inverse semigroup ${\mathcal S}(G)$ was constructed from a group $G$,
possessing a universal property related to partial representations and which was later shown by Kellendonk and Lawson \cite{KellendonkLawson:PartialActions} to be precisely the Birget-Rhodes expansion of $G$. Needless to
say, in case $\G$ is a group, the prefix expansion $\Sp$ coincides with the inverse semigroup ${\mathcal S}(G)$ constructed in \cite{Exel:PartialActionsGroupsAndInverseSemigroups}.

The interest in approaching $\Sp$ from a different route is, in part, to show that its structure, including the product operation~\eqref{LMSFormulaForProduct} above, is a direct consequence of the concept of partial
representations.
In addition, we answer some natural questions about the expansion $\Sp$ which did not take place in~\cite{LawsonMargolisSteinberg:Expansions}, as for instance what happens with $\S\SG$ as well as the order of $\SG$ if $\G$
is finite.

As already observed in \cite{LawsonMargolisSteinberg:Expansions}, the assignment $\G\mapsto \SG$ is a functor on the category $\Inv$ of inverse semigroups (with homomorphisms as morphisms). It may also be viewed as a functor
from the category $\Pinv$ of inverse semigroups with partial homomorphisms as morphisms to $\Inv$. When viewed in this way, we show that it is the left adjoint of the forgetful functor $\Inv\to \Pinv$. This is related to the
result in~\cite{Szendrei:BirgetRhodes} where Szendrei shows that $G\mapsto \SG$, when considered as a functor from the category of groups to $F$-inverse semigroups, is left adjoint of to the functor assigning to an inverse
semigroup its maximal homomorphic group image.

One of the main goals in this paper is to study the relation between Fell bundles and crossed products by $\G$ and $\SG$. We prove that Fell bundles over $\G$ correspond bijectively to \emph{saturated} Fell bundles over
$\SG$ in a somewhat canonical way. In fact, this correspondence extends to a functor and gives an equivalence between the categories of Fell bundles over $\G$ and saturated Fell bundles over $\SG$. Moreover, this equivalence
preserves the associated (full and reduced) cross-sectional \cstar{}algebras. In particular, the functor $\G\mapsto \SG$ preserves classical partial crossed products, that is, $A\rtimes \G\cong A\rtimes\SG$ for any partial
action of $\G$ on a \cstar{}algebra $A$ and the corresponding (global) action of $\SG$ on $A$. A similar result still holds for twisted partial actions. In fact, all these results are special cases of the more general
theorem we obtain for Fell bundles and their corresponding \cstar{}algebras.

\section{The expansion}

Throughout this section, we will let $\G$ be an inverse semigroup.

\begin{definition} \label{DefSG} We shall let $\SG$ denote the universal semigroup
defined via generators and relations as follows: to each element $s$ in $\G$ we
take a generator $[s]$ (from any fixed set having as many elements as $\G$), and
for every $s,t\in \G$ we consider the relations
\end{definition}
\begin{enumerate}[(i)]
  \item $[s][t][t^*] = [st][t^*]$,
  \item $[s^*][s][t] = [s^*][st]$,
  \item $[s][s^*][s] = [s]$.\label{DefSG.iii}
\end{enumerate}

The following result is inspired by \cite[Proposition 2.4]{Exel:PartialActionsGroupsAndInverseSemigroups}.

\begin{proposition}\label{PropFromInverse} For every $t$ in $\G$, let $\e t =
[t][t^*]$.  Then, for each $t$ and $s$ in $\G$,
\begin{enumerate}[(i)]
  \item $\e t$ is idempotent,
  \item $[t]\e s = \e {ts}[t]$, \label{CommutRelation}%{\PropFromInverse.\rzitem}
  \item $\e t$ and $\e s$ commute. \label{EpsilonsCommute}%{\PropFromInverse.\rzitem}
\end{enumerate}
\end{proposition}

\begin{proof} Notice that (i) follows immediately from \ref{DefSG}\eqref{DefSG.iii}.  With
respect to (iii), we have
  \begin{equation}\label{Eq:Prod.Epsilons}
  \e s\e t =
  [s][s^*] [t][t^*] =
  [s][s^*t][t^*] =
  [s][s^*t][t^*s][s^*t][t^*] =
  [ss^*t][t^*s][s^*tt^*].
  \end{equation}
  Interchanging the roles of $s$ and $t$, the above computation gives
  \begin{multline*}
  \e t\e s =
  [tt^*s][s^*t][t^*ss^*] =
  [tt^*s][s^*t][t^*s][s^*t][t^*ss^*] =
  [tt^*ss^*t][t^*s][s^*tt^*ss^*] \\
  =[ss^*tt^*t][t^*s][s^*ss^*tt^*] =
  [ss^*t][t^*s][s^*tt^*] \eq{\eqref{Eq:Prod.Epsilons}}
  \e s\e t.
  \end{multline*}
  Addressing (ii) observe that
  \begin{multline*}
  [t]\e s =
  [t][s][s^*] =
  [t][t^*][t][s][s^*] \eq{(iii)}
  [t][s][s^*][t^*][t] \\
  =[ts][s^*][t^*][t] =
  [ts][s^*t^*][t] =
  \e{ts} [t].
  \end{multline*}
\vskip -1pc
\end{proof}

In the next result, we write $E(\G)$ for the idempotent semilattice of $\G$.

\begin{proposition} \label{IdempotentsInBracket} Given $e\in E(\G)$ and $s\in\G$ one has that:
\begin{enumerate}[(i)]
  \item $\e e=[e]$,  and hence $[e]$ is idempotent,
  \item $[e][s]=[es]$, and $[s][e] = [se]$,
  \item $\e e\e s = \e{es}$.
\end{enumerate}
\end{proposition}

\begin{proof}  (i)
  This follows from
  $$
  [e] = [e][e^*][e] = [e][e^*e] = [e][e^*] =\e e.
  $$
  In order to prove (ii) we have by (i) that
  \begin{multline*}
  [e][s] = [e][e][s] = [e][es] = [e][es][(es)^*][es] =
  [ees][(es)^*][es] \\
  = [es][(es)^*][es] = [es].
  \end{multline*}
  The second part of (ii) follows similarly.
  As for (iii) we have
  $$
  \e e\e s =   \e e\e s \e e \eq{(i)}
  [e][s][s^*][e] \eq{(ii)}
  [es][s^*e] = \e{es}.
  $$
\vskip -1pc
\end{proof}

\begin{lemma}\label{Reductor} Given $s_1,\ldots,s_n,t\in \G$, let
  $x = \e {s_1} \ldots \e {s_n} [t]$.  Then
  $$
  x =   \e{s_1'}\e{s_2'}\ldots \e {s_n'}[t'],
  $$
  where $t'=pt$, $s_i' = ps_i$, for all $i=1, \ldots, n$, and
  $p = s_1s_1^*\ s_2s_2^* \ldots s_ns_n^*\ tt^*$.
\end{lemma}
\begin{proof} Letting
  $f = tt^*$
  and
  $e_i=s_is_i^*$, for all
  $i=1,\ldots,n$,
  observe that, for each $i$, one has
  $$
  \e {s_i} =   \e{e_is_i} \eq{\textup{Proposition}~\ref{IdempotentsInBracket}(iii)}
  \e{e_i}\e{s_i},
  $$
  while
  $$
  [t] = [t][t^*][t] = [tt^*][t] = [f][t] \eq{\textup{Proposition}~\ref{IdempotentsInBracket}(i)}
  \e f[t].
  $$
  So
  \begin{multline*}
  x =
  \e {s_1} \e {s_2} \ldots \e {s_n} [t] =
  \e{e_1}\e{s_1}\ \e{e_2}\e{s_2} \ldots\e{e_n} \e {s_n}\ \e{f} [t] \\
  =\e{e_1}\e{e_2}\ldots \e{e_n}\e{f}\ \e{s_1}\e{s_2}\ldots \e {s_n}[t] \eq{\textup{Proposition}~\ref{IdempotentsInBracket}(iii)}
  % \e{e_1e_2\ldots e_nf}\ \e{s_1}\e{s_2}\ldots \e {s_n}[t] \\
  \e p\e{s_1}\e{s_2}\ldots \e {s_n}[t] \\
  =\e p\e{s_1}\ \e p\e{s_2}\ldots \e p\e {s_n}\ \e p[t]  =
  \e{ps_1} \e{ps_2}\ldots \e {ps_n}[pt].
  \end{multline*}
\vskip -1pc
\end{proof}

\begin{proposition}\label{FirstCannonical} Every element $x\in\SG$ may be written
as a product
  \begin{equation}\label{Decomposition}
  x = \e {s_1} \ldots \e {s_n} [t],
  \end{equation}
  where $s_1,\ldots,s_n,t\in \G$. One may moreover assume that
\begin{enumerate}[(a)]
  \item $s_1s_1^* = \ldots = s_ns_n^* = tt^*$, and
  \item $t$ and $tt^*$ belong to the set $\{s_1,\ldots,s_n\}$.
\end{enumerate}
\end{proposition}
\begin{proof} Given $x\in\SG$ write
  $x = [r_1][r_2]\ldots[r_m]$, where the $r_i\in \G$.  We shall prove the first
sentence of the statement by induction on $m$.  The case $m=1$ follows from
  $$
  x = [r_1] = [r_1][r_1^*][r_1] = \e{r_1} [r_1].
  $$
  Assuming that $m\geq 2$, use the induction hypothesis to write
  $$
  [r_1][r_2]\ldots[r_{m-1}] = \e {s_1} \ldots \e {s_n} [t],
  $$
  with $s_1,\ldots,s_n,t\in \G$.
  Next, observe that
  \begin{equation}\label{IntermCalculation}
  [t][r_m] =  [t][t^*][t][r_m] = [t][t^*][tr_m] = \e t[tr_m].
  \end{equation}
  Therefore,
  $$
  x = [r_1][r_2]\ldots[r_m] =
  \e {s_1} \ldots \e {s_n} [t][r_m] \eq{\eqref{IntermCalculation}}
  \e {s_1} \ldots \e {s_n} \e t[tr_m],
  $$
  proving the first sentence in the statement.

In order to prove (a) write
  $x = \e {s_1} \e {s_2} \ldots \e {s_n} [t]$, using what we have already proved, and let
  $t'$, $s_i'$, and $p$, be as in Lemma~\ref{Reductor}.  Then the condition in (a)
is verified since
  $$
  s_i's_i'^* = ps_is_i^*p = p = ptt^*p = t't'^*.
  $$

  To conclude we need to check (b) but this follows from
  $$
  [t] = [t][t^*][t] = \e t[t] =  \e t[tt^*t] \eq{\textup{Proposition}~\ref{IdempotentsInBracket}(ii)}
  \e t[tt^*][t] =   \e t\e {tt^*}[t].
  $$
\vskip -1pc
\end{proof}

\begin{definition} If $x\in\SG$ is written as in \eqref{Decomposition} in
such a way that Proposition~\ref{FirstCannonical}(a--b) are satisfied, we will say that
$x$ is in \emph{normal form}.
\end{definition}

Suppose that an element $x\in\SG$ is written as
  $x = \e {s_1} \e {s_2} \ldots \e {s_n} [t]$,  as in \eqref{Decomposition},
but not necessarily in normal form.
Letting $A = \{s_1,s_2,\ldots, s_n\}$ we will denote by
  \begin{equation}\label{DefineEA}
  \e A := \e {s_1} \e {s_2} \ldots \e {s_n},
  \end{equation}
  which is unambiguously defined since the $\e {s_i}$ commute with each other by
Proposition~\ref{PropFromInverse}\eqref{EpsilonsCommute}.  Employing this notation, we may therefore write
  \begin{equation}\label{ShortDecomp}
  x = \e A[t].
  \end{equation}
Assuming now that $x$ is in normal form, notice that
Proposition~\ref{FirstCannonical}(a -- b) may be rephrased as
  \begin{equation}\label{FirstCannonicalSetLevel}
  t,tt^*\in A \quad\mbox{and}\quad ss^* = tt^*\quad \mbox{for all }s\in A.
  \end{equation}

\begin{definition}\label{DefineESet} Given $e\in E(\G)$ we will say that a subset $A\subseteq\G$ is an $e$-set if
\begin{enumerate}[(i)]
  \item $e\in A$, and
  \item $ss^*=e$, for all $s\in A$.
\end{enumerate}
\end{definition}

Observe that condition \eqref{FirstCannonicalSetLevel} may be rephrased  by
saying that $A$ is a finite $tt^*$-set containing $t$.  We may therefore
summarize our findings as follows.

\begin{proposition} Every element  $x\in \SG$ may be written as
  $$
  x = \e A[t],
  $$
  where $t\in \G$, and $A$ is a finite $tt^*$-set containing $t$,   in which case
we say that $x$ is in normal form.
\end{proposition}

Let us now give an expression for multiplying elements of $\SG$.%\footnote{\bf\large Give also an expression for the inverse of elements in $\SG$.}

\begin{proposition} \label{prop:FormulaProd} Let $A, B\subseteq \G$ be finite
subsets, let $t,s\in\G$, and let $x=\e{A}[t]$ and $y=\e{B}[s]$.  Assuming
that $t\in A$, one has that
  $$
  \begin{array}{ccc}
   xy & = & \hfill \e {A \cup (tB)}\hfill[ts] \\
          & = &  \e {(tss^*t^*A) \cup (tB)}[ts].
  \end{array}
  $$
Assuming that both $x$ and $y$ are in normal form (in which case the above
requirement that $t\in A$ is automatically satisfied) one moreover has that
the second expression for $xy$ above is in normal form.
\end{proposition}

\begin{proof} Using Proposition\ref{PropFromInverse}\eqref{CommutRelation}
it is easy to see that $[t]\e{B} =\e{(tB)}[t]$, so
  $$
  xy = \e {A}[t]\ \e{B}[s] =
  \e {A} \e{(tB)}[t][s] =
  \e {A \cup (tB)}[t][s] = \ldots
  $$
  Letting $B = A \cup (tB)$, the above equals
  $$
  \ldots =
  \e B[t][t^*][t][s] =
  \e B[t][t^*][ts] =
  \e B\e{t}[ts] =
  \e {B\cup \{t\}}[ts] =
  \e B[ts],
  $$
  proving the first expression for $xy$ above.  Next, observe that
  $$
  [ts] = [ts]  [s^*t^*]  [ts]  = [ts s^*t^*]  [ts]
=  \e q  [ts],
  $$
  where $q=ts s^*t^*$.  Therefore,
  $$
  xy = \e {A\cup(tB)}\e q[ts] =
  \e q\e {A}\e{(tB)}[ts] \eq{\textup{Proposition}~\ref{IdempotentsInBracket}(iii)} $$$$ =
  \e {(qA)}\e{(tB)}[ts] =
  \e {(qA)\cup(tB)}[ts],
  $$
  proving the second expression for $xy$ above.  Assuming that
$x_i=\e{A_i}[t_i]$ is in normal form for $i=1,2$, we have for all $s\in A$
that $ss^* = tt^*$, so
  $$
  qs(qs)^* = qss^* = ts s^*t^* \ tt^* =  ts s^*t^*,
  $$
  while for $s\in B$, one has that $ss^* = ss^*$, so
  $$
  ts(ts)^* =   tss^*t^* = tss^*t^*.
  $$
  Moreover, since $s\in B$,
  $$
  ts\in tB \subseteq (qA)\cup(tB),
  $$
  and, because $tt^*\in A$,
  $$
  tss^*t^* = tss^*t^* t t^* \in qA \subseteq (qA)\cup(tB),
  $$
  hence proving the last assertion.
\end{proof}

Proposition~\ref{FirstCannonical} may be used to prove that $\SG$ is a regular semigroup.

\begin{proposition} \label{Regular} For every $x\in \SG$ there exists $\bar
x\in\SG$ such that $x\bar xx=x$ and $\bar xx\bar x=\bar x$.
\end{proposition}
\begin{proof}
  Given $x$ in $\G$, use Proposition~\ref{FirstCannonical} to write
  $
  x = \e {s_1} \ldots \e {s_n} [t],
  $
  with $s_1,\ldots, s_n,t\in\G$,
  and set $\bar x = [t^*]\e {s_1} \ldots \e {s_n}$.
  Then, using Proposition~\ref{PropFromInverse}\eqref{EpsilonsCommute}, we get
  $$
  x\bar xx =
  \e {s_1} \ldots \e {s_n} [t] [t^*]\e {s_1} \ldots \e {s_n} \e {s_1} \ldots \e
{s_n} [t] =
  \e {s_1} \ldots \e {s_n} [t] [t^*][t] =
  \e {s_1} \ldots \e {s_n} [t] = x.
  $$
  The proof that $\bar xx\bar x=\bar x$ is just as easy.
\end{proof}

Before we proceed we need to introduce the following important concept:

\begin{definition} \label{DefPRep} Let $\T$ be a semigroup. A \emph{partial homomorphism} of $\G$ in $\T$ is a map
  $$
  \pi: \G \to \T
  $$
  such that for all $s, t\in \G$, one has that
  \begin{enumerate}[(i)]
  \item $\pi(s)\pi(t)\pi(t^*) = \pi(st)\pi(t^*)$.
  \item $\pi(s^*)\pi(s)\pi(t) = \pi(s^*)\pi(st)$,
  \item $\pi(s)\pi(s^*)\pi(s) = \pi(s)$.
  \end{enumerate}
\end{definition}

Notice that the definition of partial homomorphisms does not require $\T$ to
be an \emph{inverse} semigroup.  However, should $\T$ happen to be an inverse
semigroup, then Definition~\ref{DefPRep}(iii) applied to both $s$ and $s^*$, together
with the uniqueness of inverses, immediately implies that
  \begin{equation}\label{PreserveStar}
  \pi(s^*) = \pi(s)^* \quad\mbox{for all }s\in \G.
  \end{equation}
Hence, if $H$ is an inverse semigroup, the axioms (i)-(iii) in Definition~\ref{DefPRep} are equivalent to (i)-(ii) plus~\eqref{PreserveStar}.
If $\T$ is the semigroup $\Ls(\hils)$ of all bounded linear maps on a Hilbert space $\hils$ with composition of operators as the semigroup multiplication, partial homomorphisms $\G\to \T$ are also called \emph{partial
representations} of $\G$ on $\hils$.

It is evident that the map
  \begin{equation}\label{CanonicalMap}
  \iota_G\colon \G\to\SG,\quad s\in\G \mapsto [s]\in\SG
  \end{equation}
  is a partial homomorphism of $\G$ in $\SG$, and we will call it the
\emph{canonical partial homomorphism}.  The universal property of $\SG$ tells us
that every partial homomorphism of $\G$ factors through the canonical one.
More precisely:

\begin{proposition}\label{prop:UniProp} For every semigroup $\T$ and every partial homomorphism $\pi:\G\to\T$, there exists a unique semigroup homomorphism
$\Sh\pi:\SG\to\T$, such that the diagram
  $$
  \xymatrix{
  \G \ar[d]_{\iota_G} \ar[r]^{\pi}  & \T \\
  \SG\ar[ur]_{\Sh\pi}   &}
  $$

  commutes, where the vertical arrow is the canonical partial homomorphism.
\end{proposition}
\begin{proof} Obvious from the definition of $\SG$.
\end{proof}

This result has a very useful consequence, as follows.

\begin{corollary} \label{cor:IdempotentsInPrep}
  Let $\pi:\G\to\T$ be a partial homomorphism of $\G$ in a semigroup $\T$.
Then, given any $e\in E(\G)$ one has that
  \begin{enumerate}[(i)]
  \item $\pi(e)$ is idempotent,
  \item if $s\in \G$, then $\pi(e)\pi(s)=\pi(es)$, and $\pi(s)\pi(e) = \pi(se)$.
  \end{enumerate}
\end{corollary}
\begin{proof}
  Letting $\Sh\pi$ be as in Proposition~\ref{prop:UniProp} we have
  $$
  \pi(e)\pi(s) = \Sh\pi([e])\Sh\pi([s]) = \Sh\pi([e][s]) \eq{\textup{Proposition}~\ref{IdempotentsInBracket}(ii)}
\Sh\pi([es])  = \pi(es),
  $$
  proving the first part of (ii), while the second part follows similarly.  As
for (i), it follows from (ii) by taking $s=e$.
\end{proof}

If we take $\pi$ to be the identity map on $\G$, then it is evidently a partial homomorphism of $\G$ in itself, so Proposition~\ref{prop:UniProp} provides a semigroup
homomorphism
  \begin{equation}\label{DegreeMap}
  \partial :\SG \to \G,
  \end{equation}
  called the \emph{degree map}, such that $\partial([s]) = s$, for all $s$ in $\G$.

\begin{proposition}\label{prop:NormalFormDegree} Let $x\in\SG$ be written in normal form as
  $x = \e A [t].$ Then $\partial(x)=t$.  Moreover, $x$ is idempotent if and only if $\partial(x)=t$ is idempotent, in which case we have $x=\e A$.
  Hence $\partial$ is an \emph{essentially injective} homomorphism in the sense that pre-images of idempotents are idempotents.\footnote{Essentially injective homomorphisms are also called \emph{idempotent pure}
  homomorphisms in \cite{LawsonMargolisSteinberg:Expansions}.}
  In particular, the idempotents of $\SG$ have the form $\e A$ for some finite subset $A\sbe \G$ which can be
  chosen to be an $e$-set for some $e\in E(\G)$.
\end{proposition}
\begin{proof}  We have
  $$
  \partial(x) = \partial( \e {s_1} \ldots \e {s_n} [t]) =
  \partial([s_1][s_1^*]\ldots [s_n][s_n^*][t])=
  s_1s_1^*\ldots s_ns_n^*t =
  tt^*t = t.
  $$
In particular, if $x$ is idempotent then so is $t$. In this case, $t=\e t$ (by Proposition~\ref{IdempotentsInBracket}(i))
so that $x=\e A\e t=\e {A\cup\{t\}}=\e A$ because $A$ contains $t$.
\end{proof}

\begin{theorem} $\SG$ is an inverse semigroup.
\end{theorem}
\begin{proof}
We have already seen in Proposition~\ref{Regular} that $\SG$ is regular.  It therefore
suffices to prove that the idempotent elements of $\SG$ commute (see \cite[Chapter1, Theorem~3]{Lawson:InverseSemigroups}).
By Proposition~\ref{prop:NormalFormDegree}, every idempotent of $\SG$ has the form $\e{A}$ for some $A\sbe\G$.
The fact that idempotents commute then follows from Proposition~\ref{PropFromInverse}\eqref{EpsilonsCommute}.
\end{proof}

Recall that an inverse semigroup $\G$ is $E$-unitary if for all $s\in \G$ and $e\in E(\G)$,
$$es\in E(\G)\Rightarrow s\in E(\G).$$

\begin{proposition}
The inverse semigroup $\SG$ is $E$-unitary if and only if $\G$ is $E$-unitary.
\end{proposition}
\begin{proof}
Suppose that $\G$ is $E$-unitary, and let $x\in\SG$ and $y\in E(\SG)$ with $yx\in E(\SG)$. Then $e\defeq\partial(y)\in E(\G)$, $s\defeq\partial(x)\in \G$ and $es=\partial(yx)\in E(\G)$. Since $\G$ is $E$-unitary, we have
$\partial(x)=s\in E(\G)$, so that $x\in E(\SG)$ because $\partial$ is essentially injective.
Conversely, assume that $\SG$ is $E$-unitary and let $s\in \G$ and $e\in E(\G)$ with $es\in E(\G)$. Then $[e][s]=[es]\in E(\SG)$ (by Proposition~\ref{IdempotentsInBracket}(i)) so that $[s]\in E(\S\G)$ because $\SG$ is
$E$-unitary. Therefore, $s=\partial([s])\in E(\G)$.
\end{proof}

Before ending this section, let us remark at this point that our inverse semigroup $\SG$ is canonically isomorphic to the inverse semigroup
$$H=\{(A,t): A\sbe \G\mbox{ finite }, t,tt^*\in A \mbox{ and }ss^*=tt^*\mbox{ for all }s\in A \}$$
introduced in \cite{LawsonMargolisSteinberg:Expansions} (see our introduction). In fact, this obviously follows from the same universal property (of classifying partial homomorphisms) that both have. An explicit isomorphism
$\SG\cong H$ is given by taking an element $x=\epsilon_A[t]\in\SG$ in normal form and sending it to the pair $(A,t)\in H$. This is the same map one gets by first considering the map $\pi\colon\G\to H$, $t\mapsto
(\{tt^*,t\},t)$ and checking that $\pi$ is a partial homomorphism. The corresponding homomorphism $\tilde\pi\colon\SG\to H$ obtained from the universal property of $\SG$ is then equal to map $\epsilon_A[t]\mapsto (A,t)$ just
mentioned. In particular, this shows that this map is well-defined and, even more important, that normal forms are unique: if $x=\epsilon_A[t]=\epsilon_B[s]$ are two ways of writing $x$ in normal form, then $A=B$ and $s=t$.
Later, in section~\ref{sec:CanonicalPartialAction}, we are going to prove uniqueness of normal forms without using the inverse semigroup $H$ above: we construct a canonical partial action of $G$ that "separates" normal
forms.

\section{Partial homomorphisms and partial actions}

In Definition~\ref{DefPRep}, we defined partial homomorphisms of
inverse semigroups.  Now we will find equivalent conditions for checking that a
map into an inverse semigroup
is a partial homomorphism.  As before, we fix an inverse semigroup $\G$.

\begin{proposition} \label{prop:CharacterizationOfPartialRepresentations}
  Let $\T$ be an inverse semigroup and let $\pi:\G\to \T$ be a map.  Then
$\pi$ is a partial homomorphism if and only if, for all $s$ and $t$ in $\G$,
one has that
  \begin{enumerate}[(i)]
  \item $\pi(s^*)=\pi(s)^*$,
  \item $\pi(s)\pi(t) \leq \pi(st)$, and
  \item $\pi(s) \leq \pi(t)$ whenever $s\leq t$.
  \end{enumerate}
\end{proposition}
\begin{proof}
  Assume that $\pi$ is a partial homomorphism. Then (i) holds by
\eqref{PreserveStar}.  Next,  pick $s,t\in\G$, and   let $f =
\pi(t^*)\pi(t)$.  We then have that $f$ is idempotent by (i), and
  $$
  \pi(st) f =
  \pi(st) \pi(t^*)\pi(t) =
  \pi(s)\pi(t) \pi(t^*)\pi(t) =
  \pi(s)\pi(t),
  $$
  proving (ii). Assuming that $s\leq t$, write $s = te$, for some idempotent
$e$.  Then $\pi(e)$ is idempotent by Corollary~\ref{cor:IdempotentsInPrep}(i)
and
  $$
  \pi(t)\pi(e) \eq{\ref{cor:IdempotentsInPrep}(ii)}
  \pi(te)  = \pi(s),
  $$
  proving (iii).

Conversely,  suppose that (i) -- (iii) hold.  Given $s,t\in\G$,  we have
  \begin{multline*}
  \pi(s)\pi(t)\pi(t^*) \leq \pi(st)\pi(t^*) = \pi(st)\pi(t^*) \pi(t) \pi(t^*)\\
  \leq \pi(stt^*) \pi(t) \pi(t^*) \leq \pi(s) \pi(t) \pi(t^*),
  \end{multline*}
  where in the last step we have used (iii) and the fact that $stt^*\leq s$.
This gives Definition~\ref{DefPRep}(i),
  and \ref{DefPRep}(ii) follows similarly.
  The last axiom of Definition\ref{DefPRep} follows immediately from (i).
\end{proof}

\begin{remark}
Maps $\pi\colon\G\to \T$ satisfying conditions (i)--(iii) of Proposition~\ref{prop:CharacterizationOfPartialRepresentations} are sometimes called \emph{dual pre-homomorphisms} by
some authors (as for instance, in~\cite{LawsonMargolisSteinberg:Expansions}).
\end{remark}

Let us specialize to partial homomorphisms in symmetric inverse
semigroups, but first we would like to introduce  some terminology.

\begin{definition} Let $\G$ be an inverse semigroup and let $X$ be a set.  By a
\emph{partial action} of $\G$ on $X$ we shall mean a partial homomorphism
  $$
  \pi : \G \to \I(X),
  $$
  where $\I(X)$ denotes the symmetric inverse semigroup on $X$.%\footnote{\bf\large Require that the union of $\ran(\pi_s)=X_s$ be $X$? In the topological setting require a density condition instead? If this is not
  % satisfied, then we can replace $X$ by the (closure of the) union of $X_s$?}
\end{definition}

Given a partial action $\pi$, as above, we shall write $\pi_s$ for
$\pi(s)$, since an expression such as "$\pi_s(x)$" looks a lot nicer than
"$\big(\pi(s)\big)(x)$".

\begin{proposition}\label{RepInSimmetric} Let $X$ be a set and let $\pi: \G\to \I(X)$ be a map.  For
each $s\in \G$, let $X_s$ be the range of $\pi_s$.  Then $\pi$ is a partial
action if and only if, for all $s,t\in\G$ one has that
  \begin{enumerate}[(i)]
  \item $\pi_s\inv =  \pi_{s^*}$ (in particular this implies that the domain of $\pi_s$ coincides
with the range of $\pi_{s^*}$, namely $X_{s^*}$),
  \item $\pi_s(X_{s^*}\cap X_t) = X_s\cap X_{st}$,
  \item for every $x\in X_{t^*}\cap X_{t^*s^*}$, one has that $\pi_s\big(\pi_t(x)\big) = \pi_{st}(x)$.
  \end{enumerate}
\end{proposition}
\begin{proof}
  Initially observe that, under (i) -- (ii),  the composition ``$\pi_s\big(\pi_t(x)\big)$" appearing
in (iii) is meaningful, since
  $$
  \pi_t(x) \in \pi_t(X_{t^*}\cap X_{t^*s^*}) =
  \pi_t\big(\pi_{t^*}(X_t\cap X_{s^*})\big) =
  X_t\cap X_{s^*} \subseteq X_{s^*}.
  $$

Assume first that (i) -- (iii) hold.  In order to prove that $\pi$ is a partial
action we will use Proposition~\ref{prop:CharacterizationOfPartialRepresentations}.  Since Proposition~\ref{prop:CharacterizationOfPartialRepresentations}(i)
is granted, let us attack Proposition~\ref{prop:CharacterizationOfPartialRepresentations}(ii).  For this, let $s,t\in \G$, and
notice that the domain of $\pi_s\pi_t$ coincides with
  $$
  \pi_{t^*}(X_t\cap X_{s^*})  = X_{t^*}\cap X_{t^*s^*},
  $$
  by (ii).  Evidently this is contained in $X_{t^*s^*}$, also known
as the domain of $\pi_{st}$.  By (iii) we see that $\pi_s \pi_t$ coincides
with $\pi_{st}$ on the domain of the former, which means that
$\pi_s\pi_t\leq\pi_{st}$, proving Proposition~\ref{prop:CharacterizationOfPartialRepresentations}(ii).

Let us now study the behavior of $\pi_e$,  for an idempotent $e\in E(\G)$.
By (iii) with $s=t=e$, we deduce that for every $x\in X_e$, one has that
$\pi_e(\pi_e(x)) = \pi_e(x)$, but since $\pi_e(x)$ is injective it must be that
$\pi_e(x)=x$.  In other words, $\pi_e$ is the identity map on $X_e$.

Now let $s,t\in\G$ with $s\leq t$, so that $s = ts^*s$.
As seen above $\pi_{s^*s}$ is the identity
on $X_{s^*s}$, so
  $$
  X_{s^*s}\cap X_{t^*}=
  \pi_{s^*s}(X_{s^*s}\cap X_{t^*}) \eq{(ii)}
  X_{s^*s}\cap X_{s^*st^*} =
  X_{s^*s}\cap X_{s^*} \eq{(ii)} $$
  $$ =
  \pi_{s^*}(X_{s}\cap X_{s}) =
  \pi_{s^*}(X_{s}) =
  X_{s^*}.
  $$
  This implies that $X_{s^*}\subseteq X_{t^*}$ and,  for every $x\in X_{s^*}$,  we have
  $x\in X_{s^*s}\cap X_{s^*st^*}$, hence by (iii)
  $$
  \pi_s(x) = \pi_{ts^*s}(x) = \pi_t(\pi_{s^*s}(x)) = \pi_t(x),
  $$
  proving that $\pi_s\leq\pi_t$.

Now assume that $\pi$ is a partial action, so Proposition~\ref{prop:CharacterizationOfPartialRepresentations}(i--iii)
hold.  Then (i) is granted so let us prove (ii).  Given $s, t\in\G$,
observe that the domain of $\pi_s\pi_t$ is given by $\pi_{t^*}(X_t\cap
X_{s^*})$, which is therefore contained in the domain of $\pi_{st}$, namely
$X_{t^*s^*}$, by Proposition~\ref{prop:CharacterizationOfPartialRepresentations}(ii), that is,
  $$
  \pi_{t^*}(X_t\cap X_{s^*}) \subseteq X_{t^*s^*}.
  $$
  Since the set in the left-hand side  above is obviously also contained in
$X_{t^*}$, we deduce that
  \begin{equation}\label{OneInclusion}
  \pi_{t^*}(X_t\cap X_{s^*}) \subseteq X_{t^*}\cap X_{t^*s^*}.
  \end{equation}
  We claim that \eqref{OneInclusion} is actually an equality of sets. In order to
  prove it notice that a suitable change of variables in \eqref{OneInclusion} yields
  \begin{equation}\label{OtherInclusion}
  \pi_t(X_{t^*}\cap X_{t^*s^*}) \subseteq X_t\cap X_{tt^*s^*}.
  \end{equation}
  Evidently $stt^* \leq s$, so the domain of $\pi_{stt^*}$ is contained in the
domain of $\pi_s$ by Proposition~\ref{prop:CharacterizationOfPartialRepresentations}(iii)
or, in other words, $X_{tt^*s^*} \subseteq X_{s^*}$.  Plugging this in
\eqref{OtherInclusion} one obtains
  \begin{equation}\label{InocentInclusion}
  \pi_t(X_{t^*}\cap X_{t^*s^*}) \subseteq X_t\cap X_{s^*}.
  \end{equation}
  Applying the inverse of $\pi_t$ to both sides then gives
  $$
  X_{t^*}\cap X_{t^*s^*} \subseteq \pi_{t^*}( X_t\cap X_{s^*}),
  $$
  which happens to be precisely the converse of the inclusion in
  \eqref{OneInclusion}.  Therefore,
  $$
  \pi_{t^*}(X_t\cap X_{s^*}) = X_{t^*}\cap X_{t^*s^*},
  $$
  and our claim is proved.
Suitably changing variables in the above equality immediately yields (ii).

In order to prove (iii) let $x\in X_{t^*}\cap X_{t^*s^*}$.  Then by
\eqref{InocentInclusion} one has that $\pi_t(x)\in X_{s^*}$, so $x$ is in the
domain of $\pi_s\pi_t$, and hence by Proposition~\ref{prop:CharacterizationOfPartialRepresentations}(ii)  one has that
  $$
  \pi_s\big(\pi_t(x)\big) = \pi_{st}(x).
  $$
\vskip -1pc
\end{proof}

Let us now prove a result very similar to the above,  except that it is tailored
to require minimal effort for checking a map to be a partial action.

\begin{proposition}\label{EasyParRep} Let $X$ be a set and let $\pi: \G\to \I(X)$ be a map.  For
each $s\in \G$, let $X_s$ be the range of $\pi_s$.  Then $\pi$ is a partial
action if and only if, for all $s,t\in \G$ one has that
  \begin{enumerate}[(i)]
  \item $\pi_s\inv =  \pi_{s^*}$.
  \item $\pi_s(X_{s^*}\cap X_t) \subseteq X_{st}$,
  \item if $s\leq t$, then $X_s\subseteq X_t$,
  \item for every $x\in X_{t^*}\cap X_{t^*s^*}$, one has that $\pi_s\big(\pi_t(x)\big) = \pi_{st}(x)$.
  \end{enumerate}
\end{proposition}
\begin{proof}
  The reader should have noticed that the conditions above are very
similar to the conditions given in Proposition~\ref{RepInSimmetric}, except that the
equality in Proposition~\ref{RepInSimmetric}(ii) is replaced by the weaker inclusion in (ii)
above, while (iii) is new.

  The remark made at the beginning of the proof of
Proposition~\ref{RepInSimmetric} applies here as well,  although with a different
argument: under (i) -- (iii), let the $x\in X_{t^*}\cap X_{t^*s^*}$.  Then
  $$
  \pi_t(x) \in \pi_t(X_{t^*}\cap X_{t^*s^*}) {\buildrel {\rm (ii)} \over \subseteq}
  X_{tt^*s^*} {\buildrel {\rm (iii)} \over \subseteq}
  X_{s^*},
  $$
  hence the composition ``$\pi_s\big(\pi_t(x)\big)$" appearing in (iv) is
meaningful.

Let us first observe that conditions (i) -- (iv) above are necessary for $\pi$
to be a partial action: clearly (i), (ii), and (iv) immediately follow from
Proposition~\ref{RepInSimmetric}, while (iii) is a direct consequence of
Proposition~\ref{prop:CharacterizationOfPartialRepresentations}(iii).

Conversely, let us prove that $\pi$ is a partial action under the
assumption that the above conditions hold.  Using Proposition~\ref{RepInSimmetric},  it
is clearly enough to verify Proposition~\ref{RepInSimmetric}(ii).  Since the range of
$\pi_s$ is precisely $X_s$ one immediately deduces from (ii) that
  \begin{equation}\label{OneSodeOfTheInclusion}
  \pi_s(X_{s^*}\cap X_t) \subseteq X_s \cap X_{st}.
  \end{equation}
  Applying this with a suitable change of variables, we get
  $$
  \pi_{s^*}(X_s \cap X_{st}) \subseteq
  X_{s^*} \cap X_{s^*st} {\buildrel {\rm (iii)} \over \subseteq}
  X_{s^*} \cap X_t.
  $$
  Because $\pi_s$ is the inverse of $\pi_{s^*}$ the above implies that
  $$
  X_s \cap X_{st} \subseteq  \pi_s(X_{s^*} \cap X_t),
  $$
  which combines with~\eqref{OneSodeOfTheInclusion} to give the desired Proposition~\ref{RepInSimmetric}(ii).
\end{proof}

\section{The expansion as an adjunction of categories}

Let $\Inv$ be the category of inverse semigroups with usual homomorphisms as morphisms, and let $\Pinv$ be the category of inverse semigroups with partial homomorphisms as morphisms. Note that composition of partial
homomorphisms in again a partial homomorphism (this can be seen from Proposition~\ref{prop:CharacterizationOfPartialRepresentations}) so that $\Pinv$ is indeed a category. Given two inverse semigroups $\G$ and $\H$, we write
$\Hom(\G,\H)$ for the set of all homomorphisms $\G\to\H$ (that is, the hom-set in $\Inv$) and $\PHom(\G,\H)$ for the set of all partial homomorphisms $\G\to\H$ (that is, the hom-set in $\Pinv$). Observe that $\G\mapsto \SG$
is a functor $\Pinv\to\Inv$: given a partial homomorphism $\pi\colon G\to H$ the induced homomorphism $\S\pi\colon \SG\to \S\H$ is given by $S(\pi)=\iota_\H\circ \Sh\pi$, where $\iota_\H\colon \H\to \S\H$ is the canonical
inclusion $h\mapsto [h]$ and $\Sh\pi\colon \SG\to H$ is the morphism provided by Proposition~\ref{prop:UniProp}. In other words, $\S\pi[g]=[\pi(g)]$ for all $g\in \G$.
By Proposition~\ref{prop:UniProp}, there is a bijective correspondence between partial homomorphisms $\G\to \H$ and homomorphisms $\SG\to\H$, that is, we have a canonical bijection $\Hom(\G,\H)\cong\PHom(\SG,\H)$. It is easy
to see that this correspondence is natural in both variables $\G$ and $\H$. This already says that we have an adjunction between the prefix expansion functor $\Pr\colon\Pinv\to\Inv$, $\G\to\SG$ and the forgetful functor
$\Fg\colon\Inv\to\Pinv$.

Now we describe the above adjunction in terms of units and counits. Given an inverse semigroup $\G$, we write $\partial_G\colon \SG\to\G$ for the degree map and $\iota_\G\colon\G\to\SG$ for the canonical inclusion map, that
is, $\partial_G([g])=g$ and $\iota_G(g)=[g]$ for all $g\in \G$. Then $\partial$ may be interpreted as a natural transformation from the functor $\Pr\circ\Fg\colon \Inv\to \Inv$ to the identity functor $\Id_{\Inv}$ on $\Inv$,
and $\iota$ may be viewed as a natural transformation from the identity functor $\Id_{\Pinv}$ to the functor $\Fg\circ\Pr\colon\Pinv\to\Pinv$.
Moreover, $(\iota,\partial)$ is the pair unit-counit for the adjunction between $\Pr$ and $\Fg$. To see this, one has to check the equations
$$1_{\Pr}=\partial\Pr\circ\Pr\iota\quad\mbox{and}\quad 1_{\Fg}=\Fg\partial\circ\iota\Fg,$$
write $1$ denotes identity natural transformations. More explicitly, this means that the composition
$$\partial_{\SG}\circ\Pr(\iota_G)\colon\SG\to\S\SG\to\SG$$
is the identity homomorphism $\Id_{\SG}\colon\SG\to\SG$; and the composition
$$\partial_G\circ\iota_G\colon \G\to\SG\to\G$$
is the identity (partial) homomorphism $\Id_{\G}\colon \G\to\G$. Both assertions are easily checked.

\section{The canonical partial action}\label{sec:CanonicalPartialAction}

We again fix an inverse semigroup $\G$.  The goal of this section is to exhibit
a somewhat canonical partial action of $\G$ which will, among other
things, enable us to prove that the normal form of each element in $\SG$ is
unique.

\begin{definition} A nonempty subset $\xi\subseteq \G$ will be called a
\emph{filter} if for every $e\in E(\G)$ and $s\in\G$,
  \begin{equation}\label{ParFilCondition}
  es\in\xi \iff e\in\xi \quad\mbox{and}\quad\ s\in\xi.
  \end{equation}
\end{definition}

If $G$ is a group, a filter is just a subset $\xi\sbe \G$ containing $1$ (the unit of $\G$).
In general, filters may contain several idempotents (for instance $\G$ is always a filter).
If $\G$ has a zero $0\in \G$, then the only filter containing $0$ is $\G$. Also, observe that $E(S)$ is a filter if and only if $S$ is $E$-unitary.

\begin{remark}
In general, a filter does not satisfy the symmetric property of~\eqref{ParFilCondition}:
$$se\in \xi\iff e\in \xi\quad\mbox{and}\quad\ s\in\xi.$$
In fact, it may happen that $s\in \xi$ but $s^*s\notin\xi$ and also $s^*\notin\xi$.
As a simple example, let $\G$ be the inverse semigroup with $0$ and one generator $s\not=0$ satisfying $s^2=0$. This is the inverse semigroup with five elements $G=\{0,e,f,s,t\}$, where $t=s^*$, $ss^*=e$ and $s^*s=f$ (with
$s^2=t^2=0$). Observe that $\xi=\{e,s\}$ is a filter in $\G$. We have $s\in \xi$ but $s^*=t\notin \xi$ and $s^*s=f\notin\xi$. The same problem happens for the filter $\eta=\{f,t\}$.

Let us use the opportunity to mention that $\SG=\{\epsilon_0,\epsilon_e,\epsilon_s,\epsilon_f,\epsilon_t,[s],[t]\}$, an inverse semigroup containing $\{\epsilon_0=[0],\epsilon_s,\epsilon_t,[s],[t]\}$ as an inverse
subsemigroup isomorphic to $\G$ and two other new idempotents $\epsilon_e\geq \epsilon_s$ and $\epsilon_f\geq \epsilon_t$.
\end{remark}

Here is a useful alternative characterization of filters.

\begin{proposition} \label{ParfilConds} Let  $\xi$ be a nonempty subset of $\G$.
Then $\xi$ is a filter if and only if
  \begin{enumerate}[(i)]
  \item $ss^*\in\xi$ for every $s\in\xi$;
  \item $ss^*t\in\xi$ for every $s,t\in\xi$; and
  %  \item if $s\in\xi$, $t\in \G$ and $s\leq t$, then $t\in\xi$.
  \item whenever $s, t\in \G$ are such that $s\in \xi$ and $s\leq t$, we have $t\in\xi$.
  \end{enumerate}
\end{proposition}
\begin{proof}
  Assuming that $\xi$ is a filter, let $s\in\xi$.  With $e=ss^*$, we have
 $es = s \in\xi$, hence $e\in\xi$,   proving (i).  If one is
also given $t\in\xi$, then  $ss^*t =
et\in\xi$,  proving (ii).
Next, assume that $t\geq s\in\xi$.  Again with $e = ss^*$, we have
  $
  et = ss^*t = s \in\xi,
  $
  so~\eqref{ParFilCondition} implies that $t\in\xi$.

Conversely, suppose that $\xi$ satisfies (i)--(iii) and assume that
$es\in\xi$, where $e$ is idempotent. Then
  $$
  \xi \ni es \leq s
  $$
  and hence $s\in\xi$, by (iii).  Moreover,  by (i),
  $$
  \xi \ni \ es(es)^* = ess^* \leq e,
  $$
  and hence $e\in\xi$, again by (iii).
  On the other hand, assuming that $e,s\in\xi$, with $e$ idempotent, we have by
(ii) that
  $$
  es =   ee^*s \in \xi.
  $$
\vskip -0,5cm
\end{proof}

\begin{definition} \label{DefineParFilBase} A nonempty subset $\xi\subseteq\G$ is said to be
a \emph{filter base} if it satisfies Proposition~\ref{ParfilConds}(i--ii).
\end{definition}

Observe that Proposition~\ref{ParfilConds}(i) follows from Proposition~\ref{ParfilConds}(ii) by taking $s=t$. So a filter base is just a subset $\emptyset\not= \xi\sbe \G$ satisfying Proposition~\ref{ParfilConds}(ii).

As a relevant example, notice that every $e$-set (Definition~\ref{DefineESet}) is a filter base.

\begin{proposition} \label{FromBaseToFilter} Let $\eta$ be a filter base.  Then
  $$
  \gen\eta := \{t\in\G: t\geq s, \hbox{ for some } s\in \eta\}
  $$
  is a filter.
\end{proposition}
\begin{proof}
  Let $e\in E(\G)$ and $s\in\G$.  Supposing that $es\in\gen\eta$, there exists $t\in \eta$
such that $es \geq t$.    But then
  $$
  \eta\ni  t \leq es \leq s
  $$
  from where we see that $s\in \gen\eta$.  Moreover
  $$
  \eta \ni tt^* \leq es(es)^*  = ess^* \leq e,
  $$
  so $e\in  \gen\eta$.
  Conversely, suppose that both $e$ and $s$ lie in $\gen\eta$. Then there are
$e_1, s_1\in\eta$ such that $e_1\leq e$, and $s_1\leq s$,  so
  $$
  \eta\ni e_1e_1^*s_1 \leq ee^*s = es,
  $$
  and  therefore $es\in\gen\eta$.
\end{proof}

\begin{definition}
  If $\eta$ is a filter base, we will refer to $\gen\eta$ as the \emph{filter generated by
$\eta$}.  If $\xi$ is a filter and $\gen\eta=\xi$, we will say that $\eta$ is a
\emph{filter base for} $\xi$.
\end{definition}

If $\xi$ is a filter,  then it is evidently also a filter base and $\gen\xi=\xi$.

\begin{definition}
  We will denote by $X$ the collection of all filters,  and for $t\in\G$ we will set
  $$
  X_t = \{\xi\in X: t\in\xi\}.
  $$
\end{definition}

We next start preparing for the construction of our partial action of $\G$ on $X$.

\begin{proposition} \label{FilBasesAndAllThat} If \ $t\in\G$ and $\eta$ is a
filter base containing $t^*$, then
  $$
  t\eta:= \{ts: s\in\eta\}
  $$
  is a filter base containing $t$.  If moreover $\eta'$ is another
filter base containing $t^*$, such that $\gen\eta = \gen{\eta'}$, then
$\gen{t\eta}=\gen{t\eta'}$.
\end{proposition}
\begin{proof}
  Given $u\in t\eta$, write $u=ts$, with $s\in\eta$.  Then
  $
  ss^*t^*\in\eta,
  $
  so
  $$
  uu^* = tss^*t^* \in t\eta.
  $$
  If $v$ is another element of $t\eta$, write $v = tr$, with $r\in\eta$, and notice
that $ss^*r\in\eta$, so
  $$
  uu^*v = tss^*t^*tr = tt^*tss^*r =  tss^*r \in t\eta.
  $$
  This proves that $t\eta$ is a filter base.  Since $t^*\in\eta$, we have
$t^*t\in\eta$, so
  $$
  t = tt^*t \in t\eta.
  $$

Given $\eta'$ as above, suppose that $u\in\gen{t\eta}$.  Then there exists
$s\in\eta$ such that $u\geq ts$.  Since
  $$
  s\in \eta \subseteq \gen\eta = \gen{\eta'},
  $$
  there exists $s'\in\eta'$ such that $s\geq s'$.  Therefore,
  $$
  u \geq ts \geq ts' \in t\eta',
  $$
  proving that $u\in \gen{t\eta'}$.  Thus shows that $\gen{t\eta}\subseteq \gen{t\eta'}$, and the
reverse inclusion may be proved in a similar way.
\end{proof}

We are now ready to introduce the partial action which is the main object of this section.

\begin{theorem} \label{theo:DefineMainRep} For each $t\in \G$, define
  $$
  \pi_t : X_{t^*} \to X_t,
  $$
  by $\pi_t(\xi) = \gen{t\xi}$, for all $\xi\in X_{t^*}$.  Then $\pi$ is a
partial action of\/ $\G$ on $X$.
\end{theorem}
\begin{proof} We will check the conditions in Proposition~\ref{EasyParRep}.
  For this, we first claim that $\pi_{s^*}\pi_s$ is the identity on $X_{s^*}$,
for every $s$ in $\G$.  Given $\xi\in X_{s^*}$, we have $s\in s\xi$, and
$s\xi$ is a filter base for $\pi_s(\xi)$, so we deduce from
Proposition~\ref{FilBasesAndAllThat} that $s^*s\xi$ is a filter base for
$\pi_{s^*}(\pi_s(\xi))$.  Thus, in order to prove our claim, it suffices to show
that
  $
  \gen{s^*s\xi} = \xi.
  $
  As $s^*\in\xi$, one has that $s^*s\xi\subseteq\xi$, and hence   $\gen{s^*s\xi} \subseteq
  \xi$.  Conversely, given $u\in\xi$, we have
  $$
  u \geq s^*su \in s^*s\xi,
  $$
  so $u \in\gen{s^*s\xi}$, proving that
$\gen{s^*s\xi} = \xi$, and therefore that $\pi_{s^*}(\pi_s(\xi)) = \xi$.  Reversing
the roles of $s$ and $s^*$ we deduce that $\pi_s\pi_{s^*}$ is the identity on
$X_s$, whence $\pi_{s^*}=\pi_s\inv$, proving Proposition~\ref{EasyParRep}(i).

In order to
prove Proposition~\ref{EasyParRep}(ii), let $t,s\in\G$ and let $\xi\in X_{s^*}\cap X_t$.
This means that $s^*,t\in\xi$, so evidently
  $$
  st \in s\xi \subseteq \gen{s\xi} = \pi_s(\xi),
  $$
  which implies that $\pi_s(\xi)\in X_{st}$.

With respect to Proposition~\ref{EasyParRep}(iii), assume that $s\leq t$, and pick $\xi\in
X_s$.  Then
  $$
  t \geq s  \in \xi,
  $$
  and hence $t\in\xi$, by Proposition~\ref{ParfilConds}(iii), proving that $\xi\in X_t$.

Finally, assuming that $\xi\in X_{t^*}\cap X_{t^*s^*}$, we claim that
  $$
  \gen{st\xi} =   \gen{s\gen{t\xi}}.
  $$
  Clearly $t\xi \subseteq \gen{t\xi}$, so $st\xi \subseteq s\gen{t\xi}$, and hence
$\gen{st\xi} \subseteq \gen{s\gen{t\xi}}$.  Conversely, given $u\in \gen{s\gen{t\xi}}$,
there exists $v\in \gen{t\xi}$ such that
  $
  u \geq sv,
  $
  and there exists $w\in\xi$, such that
  $
  v\geq tw.
  $
  So
  $$
  u \geq sv \geq stw \in st\xi,
  $$
  implying that $u\in\gen{st\xi}$.  This proves our claim and hence
  $$
  \pi_s(\pi_t(\xi)) =   \pi_s(\gen{t\xi})  = \gen{s\gen{t\xi}} = \gen{st\xi} = \pi_{st}(\xi).
  $$
The reader may have thought of another strategy to prove this last fact, using
$t\xi$ as a filter base for $\pi_t(\xi)$ when computing
  $
  \pi_s(\pi_t(\xi)).
  $
  However, there is no guarantee that  $s^*\in t\xi$, so this might not work.
\end{proof}

We shall now employ the above partial action to deduce the promised fact that the
normal form of each element in $\SG$ is
unique.   Recall that every $e$-set is a filter base, so if $A$ is an
$e$-set we may form the filter $\gen A$ generated by $A$.

\begin{lemma} \label{lem:WhatItContains} Let $A$ be an $e$-set, where $e\in
E(\G)$, and let $\xi=\gen A$.  Then
  $$
  A = \{s\in \xi: ss^*=e\}.
  $$
\end{lemma}
  \begin{proof}
    It is evident that $A$ is contained in the set in the right-hand side.
Conversely, given $s\in\xi$, such that $ss^*=e$, there exists $t\in A$,
such that $s\geq t$.
  Thus,
  $$
  t = tt^*s = es = ss^*s = s,
  $$
  which implies that $s\in A$.
\end{proof}

\begin{theorem}\label{theo:UniquenessOfNormalForm}
Let $x,y\in \SG$ be written in normal form as $x = \e A[s]$
and $y = \e B[t]$.  If $x=y$, then $s=t$ and $A=B$.
\end{theorem}
\begin{proof}
  By Proposition~\ref{prop:NormalFormDegree}, we have
  $$
  s = \partial(x) = \partial(y) = t,
  $$
  so $s=t$.
  Since $s\in A$, and $s=t\in B$, we deduce that
  $$
  \e A = \e A \e s =
  \e A[s][s^*] =   x[s^*] =  y[s^*] = \e B[s][s^*] = \e B\e s = \e B,
  $$
  meaning that $\e A=\e B$.
  Next, consider the semigroup homomorphism
  $$
  \Sh\pi:\SG\to\I(X),
  $$
  provided by Proposition~\ref{prop:UniProp} in terms of the partial action $\pi$ introduced in
  Theorem~\ref{theo:DefineMainRep}.
  For every $s\in\G$, we have
  $$
  \Sh\pi(\e s) =
  \Sh\pi([s][s^*]) =
  \Sh\pi([s])  \Sh\pi([s^*]) =
  \pi_s \pi_{s^*} = \pi_s\pi_s\inv = {id}_{X_s}.
  $$
  Therefore,
  $$
  \Sh\pi(\e A) =
  \pros_{s\in A} \Sh\pi(\e s) =
  \pros_{s\in A} {id}_{X_s} = id_{X_A},
  $$
  where by  ${X_A}$ we of course mean $\bigcap_{s\in A} X_s$. Having seen above
that $\e A=\e B$, we then deduce that $X_A = X_B$.

Since $\e A[s]$ and $\e B[s]$ are in normal form we have by definition
that $A$ and $B$ are $e$-sets, where $e= ss^*$.
  Recalling that every $e$-set is a filter base, we have by Proposition~\ref{FromBaseToFilter}
that $\xi:= \gen A$ is a filter, which evidently belongs to $X_A$,  and hence also
to $X_B$,  meaning that $B\subseteq\xi$.
  Using Lemma~\ref{lem:WhatItContains}, we then obtain
  $$
  A = \{s\in \xi: ss^*=e\} \supseteq B,
  $$
  and a symmetric argument yields $A\subseteq B$, proving that $A=B$ and
completing the proof.
\end{proof}

One natural question is whether $\S\SG\cong\SG$. The following result shows that this happens if and only if $\G$ is a semilattice.

\begin{proposition}
Let $\G$ be an inverse semigroup. Then the following assertions are equivalent:
\begin{enumerate}[(i)]
% \item there is an isomorphism $\SG\cong G$;
\item every partial homomorphism of $\G$ is a homomorphism;
\item the canonical map $\iota_G\colon g\mapsto [g]$ from $\G$ into $\SG$ is a homomorphism;
\item $\{[g]: g\in \G\}$ is an inverse subsemigroup of $\SG$ \textup(and hence equals $\SG$\textup);
\item the canonical map $\G\to \SG$ is an isomorphism \textup(whose inverse is the degree map $\partial\colon\SG\to \G$\textup);
\item $\G=E(\G)$, that is, $\G$ is a semilattice.
\end{enumerate}
In particular, every inverse semigroup $\G$ which is not a semilattice has a partial homomorphism which is not a representation.
Moreover, in this case we have a strictly increasing chain of inclusions:
$$\G\into\SG\into\S\SG\into\ldots$$
\end{proposition}
\begin{proof}
The implication (i)$\Rightarrow$(ii) is obvious since the canonical map $\iota_G\colon\G\to \SG$ is a partial homomorphism.
On the other hand, if the canonical map $g\mapsto [g]$ is a homomorphism, then $\{[g]: g\in \G\}$ is an inverse subsemigroup of $\SG$ which generates it, so that $\SG=\{[g]: g\in \G\}$. In this case, $\iota_G$ must be an
isomorphism from $\G$ to $\SG$ because
$\partial\colon \SG\to \G$ is a homomorphism satisfying $\partial([g])=g$ for all $g\in \G$. Hence, we have the implications (ii)$\Rightarrow$(iii)$\Rightarrow$(iv). To prove (iv)$\Rightarrow$(v), suppose that $g\mapsto [g]$
is a homomorphism of $\G$. Then we have $\epsilon_g=[g][g^*]=[gg^*]=\epsilon_{gg^*}$.
Observe that the map $g\mapsto\e g$ is injective. In fact, the normal form of $\e g$ is $\e g \e{gg^*}[gg^*]$. Thus, if $\e g=\e h$, then
$\{g,gg^*\}=\{h,hh^*\}$ (by Theorem~\ref{theo:UniquenessOfNormalForm}) so that $g=h$. In particular,
since $\epsilon_g=\epsilon_{gg^*}$, it follows that $g=gg^*$ for all $g\in \G$, that is, $\G=E(\G)$.
Finally, to check (v)$\Rightarrow$(i), suppose that $\G$ is a semilattice and $\pi$ is a partial homomorphism of $\G$.
Then  $\pi(ef)=\pi(e)\pi(f)$ for all $e,f\in E(\G)=\G$ by Corollary~\ref{cor:IdempotentsInPrep}, so that $\pi$ is a homomorphism.
\end{proof}

\begin{proposition}
Let $\G$ be a finite inverse semigroup. Given $e\in E(\G)$, define $\G^e\defeq\{s\in \G\st ss^*=e\}$ and let $p_e=|\G^e|$ be the number of elements
of $\G^e$. Then $\SG$ has exactly
\begin{equation}\label{eq:OrderOfE(S(G))}
|E(\SG)|=\sum\limits_{e\in E(\G)}2^{p_e-1}
\end{equation}
idempotent elements and exactly
\begin{equation}\label{eq:OrderOfS(G)-E(S(G))}
|\SG-E(\SG)|=\sum\limits_{e\in E(\G)}2^{p_e-2}(p_e-1)
\end{equation}
non-idempotent elements. In particular,
the total number of elements of $\SG$ is given by
\begin{equation}\label{eq:OrderOfS(G)}
|\SG|=\sum\limits_{e\in E(\G)}2^{p_e-1}+\sum\limits_{e\in E(\G)}2^{p_e-2}(p_e-1)=\sum\limits_{e\in E(G)}2^{p_e-2}(p_e+1).
\end{equation}
\end{proposition}
\begin{proof}
Given $e\in E(\G)$, let $X^e=\{A\sbe\G^e\st e\in A\}$. Observe that $|X^e|=2^{p_e-1}$.
We already know (see~Proposition\ref{prop:NormalFormDegree})
that the idempotents of $\SG$ have the form $\epsilon_A$ for some $e\in E(\G)$ and some $A\in X^e$.
And by Theorem~\ref{theo:UniquenessOfNormalForm}, for $A,B\sbe \G^e$, we have $\epsilon_A=\epsilon_B$ if and only if $A=B$ (observe that the normal form of $\epsilon_A$ is $\epsilon_A[e]$ whenever $A\in X^e$). This
implies~\eqref{eq:OrderOfE(S(G))}.
To prove \eqref{eq:OrderOfS(G)-E(S(G))}, observe that (by Theorem~\ref{theo:UniquenessOfNormalForm})
every non-idempotent element of $\SG$ can be uniquely written as
$\epsilon_A[s]$ for some $A\sbe X^e$ and some $s\in A$ with $s\not=e$. Notice that we have exactly $2^{p_e-2}$ sets $A$ in $X^e$ with this property. Since we have $p_e-1$ elements $s\in \G^e$ with $s\not=e$,
Equation~\eqref{eq:OrderOfS(G)-E(S(G))} follows as well.
\end{proof}

\section{Twisted partial actions on $C^*$-algebras}

In this section, we are going to enrich sets with further structure and require a partial action of an inverse semigroup $\G$
to be compatible with this structure. For instance, if $X$ is a topological space, and $\alpha\colon\G\to \I(X)$ is a partial action
of $\G$ on $X$, then it is natural to require the domain $\D_{s^*}\defeq \dom(\alpha_s)$ and the range $\D_s\defeq\ran(\alpha_s)$
of $\alpha_s$ to be open subsets of $X$ and $\alpha_s\colon\D_{s^*}\to \D_s$ to be continuous for all $s\in \G$.
Observe that this implies that each $\alpha_s$ is a homeomorphism with inverse $\alpha_{s^*}$. Alternatively, we could say that $\alpha_s$ is
a partial homeomorphism of $X$. So, a partial action of $\G$ on a topological space $X$ is just a partial homomorphism of $\G$ into the
inverse semigroup of all partial homeomorphisms of $X$. We also require that the union $\cup\D_s$ be dense in $X$ -- otherwise this can be arranged replacing $X$ by the closure of $\cup\D_s$.

Similarly, if $B$ is a \cstar{}algebra and $\beta\colon\G\to \I(B)$ is a partial action of $\G$ on $B$, it is natural to require
each $\beta_s\colon\D_{s^*}\to \D_s$ to be a \emph{partial automorphism} between (closed, two-sided) ideals $\D_{s^*},\D_s$ of $B$.
Of course, by a partial automorphism of $B$ we mean a \Star{}isomorphism $I\congto J$ between ideals $I,J$ of $B$.
Hence, a partial action of $\G$ on a \cstar{}algebra $B$ is just a partial homomorphism from $\G$ into the inverse semigroup of
all partial automorphisms of $B$. In addition, we require that the union $\cup\D_s$ spans a dense subspace of $B$.
If $B=\contz(X)$ is a commutative \cstar{}algebra, then partial automorphisms of $B$ correspond to partial homeomorphisms of $X$, so
partial actions on $B$ correspond to partial actions on $X$.

The notion of partial actions of inverse semigroups on \cstar{}algebras appears in \cite{Sieben:crossed.products}.
In the case of groups, it can be seen as a special case of the twisted partial actions defined by the second named author in \cite{Exel:TwistedPartialActions}. In \cite{BussExel:Regular.Fell.Bundle}, we considered a notion
of twisted (global) actions for inverse semigroups improving Sieben's definition in \cite{SiebenTwistedActions}.
We already know that partial actions of $\G$ correspond bijectively to actions of $\SG$.
A similar result should also hold with twists, that is, \emph{twisted partial actions} of $\G$ should correspond to twisted (global) actions of $\SG$.
In order to prove this, we first have to find the right definition of twisted partial actions in the realm of inverse semigroups.
Using the canonical partial homomorphism of $\G$ in $\SG$ to view $\G$ as a subset of $\SG$, a short way to define
this is to say that a twisted partial action of $\G$ is the restriction of a twisted action of $\SG$ (in
the sense of \cite[Definition~4.1]{BussExel:Regular.Fell.Bundle}) to $\G$. However,
since this concept should be a simultaneous generalization of \cite[Definition~4.1]{BussExel:Regular.Fell.Bundle} and
\cite[Definition~2.1]{Exel:TwistedPartialActions}, it is not so difficult to guess what should be the right definition:

Given a \cstar{}algebra $B$, we write $\M(B)$ for the multiplier algebra of $B$ and $\U\M(B)$ for the group of unitary multipliers of $B$.

\begin{definition}\label{def:TwistedPartialAction}
Let $\G$ be an inverse semigroup and let $B$ be a \cstar{algebra}. A \emph{twisted partial action} of $\G$ on $B$ is
a pair $(\beta,\omega)$, where $\beta=\{\beta_s\}_{s\in\G}$ is a family of partial automorphisms $\beta_s\colon\D_{s^*}\to \D_s$
of $B$ and $\omega=\{\omega(s,t)\}_{s,t\in \G}$ is a family of unitary multipliers $\omega(s,t)\in \U\M(\D_s\cap\D_{st})$, such that
the linear span of the family of ideals $\{\D_s\}_{s\in \G}$ is dense in $B$ and
the following conditions are satisfied for all $r,s,t\in \G$ and $e,f\in E(\G)$:
\end{definition}
\begin{enumerate}[(i)]
\item $\beta_r(\D_{r^*}\cap\D_s)=\D_r\cap \D_{rs}$;
\item $\beta_r(\beta_s(x))=\omega(r,s)\beta_{rs}(x)\omega(r,s)^*$ for all $x\in \D_{s^*}\cap\D_{s^*r^*}$;
        \label{def:twisted action:item:beta_r beta_s=Ad_omega(r,s)beta_rs}
\item $\beta_r(x\omega(s,t))\omega(r,st)=\beta_r(x)\omega(r,s)\omega(rs,t)$ whenever $x\in\D_{r^*}\cap\D_s\cap\D_{st}$;
        \label{def:twisted action:item:CocycleCondition}
\item $\omega(e,f)=1_{ef}$ and $\omega(r,r^*r)=\omega(rr^*,r)=1_{r}$, where $1_r$ is the unit of $\M(\D_r)$;
        \label{def:twisted action:item:omega(r,r*r)=omega(e,f)=1_ef_and_omega(rr*,r)=1}
\item $\omega(s^*,e)\omega(s^*e,s)x=\omega(s^*,s)x$ for all $x\in \D_{s^*e}$.
        \label{def:twisted action:item:omega(s*,e)omega(s*e,s)x=omega(s*,s)x}
\end{enumerate}

It is not difficult to see that the above definition is in fact a generalization of both \cite[Definition~4.1]{BussExel:Regular.Fell.Bundle} and
\cite[Definition~2.1]{Exel:TwistedPartialActions}. Let us observe that the first axiom above is not included in the definition
of twisted (global) actions appearing in \cite[Definition~4.1]{BussExel:Regular.Fell.Bundle}, but it is a consequence of the other axioms.
In fact, if $(\beta,\omega)$ is a twisted action, then by \cite[Lemma~4.6]{BussExel:Regular.Fell.Bundle} we have
$\beta_r(\D_{r^*}\cap\D_s)=\D_{rs}=\D_r\cap\D_{rs}$ because (by the same lemma) $\D_{rs}=\D_{rss^*r^*}\sbe \D_{rr^*}=\D_r$.

\begin{remark}
Axiom (ii) in the above definition is equivalent to
\begin{equation}\label{eq:CompositionOfBetas}
\beta_r\circ \beta_s=\Ad_{\omega(r,s)}\circ\beta_{rs}\quad\mbox{(composition of partial automorphisms)},
\end{equation}
where $\Ad_{\omega(r,s)}\colon\D_r\cap\D_{rs}\to \D_r\cap\D_{rs}$ is the partial automorphism defined by $\Ad_{\omega(r,s)}(z)=\omega(r,s)z\omega(r,s)^*$. In fact, observe that (by the properties (iii) and (iv) to be proved
in Proposition~\ref{prop:PropertiesTwistedPartialAction} below)
$$\dom(\beta_r\circ\beta_s)=\beta_{s}\inv(\D_s\cap\D_{r^*})=\D_{s^*}\cap\D_{s^*r^*}=\D_{s^*r^*r}\cap\D_{s^*r^*}.$$
On the other hand,
$$\dom(\Ad_{\omega(r,s)}\circ\beta_{rs})=\beta_{rs}\inv(\D_{rs}\cap\D_r)=\D_{s^*r^*}\cap\D_{s^*r^*r}.$$
Thus, $\dom(\beta_r\circ\beta_s)=\dom(\Ad_{\omega(r,s)}\circ\beta_{rs})$. In the untwisted case (that is, if $\omega(r,s)$ is the unit multiplier of $\D_r\cap\D_{rs}$ for all $r,s$), we may interpret
$\Ad_{\omega(r,s)}\circ\beta_{rs}$ as the restriction of $\beta_{rs}$ to the domain of $\beta_r\circ\beta_s$. In this case, Equation~\eqref{eq:CompositionOfBetas} just says that $\beta_{rs}$ extends $\beta_r\circ\beta_s$.
\end{remark}

Here are some properties of twisted partial actions (compare with \cite[Lemma~4.6]{BussExel:Regular.Fell.Bundle}).

\begin{proposition}\label{prop:PropertiesTwistedPartialAction}
Let $(\beta,\omega)$ be a twisted partial action of $\G$ on $B$ as above. Then the following holds:
\begin{enumerate}[(i)]
\item  $\D_s\sbe \D_{ss^*}$ for all $s\in \G$. Moreover, $\D_s=\D_{ss^*}$ for all $s\in \G$ if and only if $(\beta,\omega)$ is a twisted
(global) action in the sense of \cite[Definition~4.1]{BussExel:Regular.Fell.Bundle};
\item $\beta_e\colon\D_e\to \D_e$ is the identity map for all $e\in E(\G)$;
\item $\D_s\cap\D_t=\D_s\cap\D_{ss^*t}=\D_{tt^*s}\cap\D_{ss^*t}=\D_{es}\cap\D_{et}$ for all $s,t\in \G$, where $e=ss^*tt^*$; moreover $\D_e\cap\D_s=\D_{es}$ for all $e\in E$ and $s\in
    S$.\label{prop:PropertiesTwistedPartialAction:RelationBetweenDomains}
\item $\beta_s\inv(\D_s\cap\D_{r^*})=\D_{s^*}\cap \D_{s^*r^*}$;
\item $s\leq t \Rightarrow \D_s\sbe\D_t$;
\end{enumerate}
\end{proposition}
\begin{proof}
(i) Observe that $\D_s=\beta_s(\D_{s^*})=\beta_s(\D_{s^*}\cap \D_{s^*})=\D_s\cap\D_{ss^*}$, so that $\D_s\sbe\D_{ss^*}$ for all $s\in \G$.
If $\D_{s}=\D_{ss^*}$ for all $s\in \G$, we have to check that $(\beta,\omega)$ is a twisted global action, that is, the axioms (i)-(iv) appearing in \cite[Definition~4.1]{BussExel:Regular.Fell.Bundle} are satisfied. But the
axioms are essentially the same by looking at property (iii) we are going to prove below.

(ii) By axiom (iii) in Definition~\ref{def:TwistedPartialAction}, we have $\omega(e,e)=1_e$, so that
$\beta_e(\beta_e(x))=\omega(e,e)\beta_e(x)\omega(e,e)=\beta_e(x)$ for all $\D_e$ by Definition~\ref{def:TwistedPartialAction}(ii).
Since $\beta_e\colon\D_e\to \D_e$ is an isomorphism, this implies that $\beta_e=\id_{\D_e}$.

(iii) First, given $f\in E(\G)$, since $\beta_f=\id_{\D_f}$, we have $\D_f\cap\D_t=\beta_f(\D_f\cap\D_t)=\D_f\cap\D_{ft}$ by Definition~\ref{def:TwistedPartialAction}(i). Since $\D_s\sbe\D_{ss^*}$, taking $f=ss^*$, we get
$\D_s\cap\D_t=\D_s\cap\D_{ss^*}\cap\D_t=\D_s\cap\D_{ss^*t}$. Now, applying this again, we get $\D_s\cap\D_{ss^*t}=\D_{ss^*t}\cap\D_s=
\D_{ss^*t}\cap\D_{ss^*tt^*s}=\D_{ss^*t}\cap\D_{tt^*s}=\D_{es}\cap\D_{et}$.

(iv) By Definition~\ref{def:TwistedPartialAction}(i) and (iii) above, we have $\beta_s(\D_{s^*}\cap\D_{s^*r^*})=\D_s\cap\D_{ss^*r^*}=\D_s\cap\D_{r^*}$, whence (iv) follows.

(v) If $s\leq t$, then $s=ss^*t$, so that $\D_s\cap\D_t\eq{(iii)}\D_s\cap\D_{ss^*t}=\D_s$. Hence $\D_s\sbe\D_t$.
\end{proof}

Let $(\beta,\omega)$ be a twisted partial action of an inverse semigroup $S$ on a \cstar{}algebra $B$ as in Definition~\ref{def:TwistedPartialAction}. Given an element $x=\epsilon_{r_1}\cdots\epsilon_{r_n}[r]$ of $\SG$ in
normal form, we define $\tilde\D_x\defeq \D_{r_1}\cdots\D_{r_n}\cdot\D_r$ and
$\tilde\beta_x\colon\tilde\D_{x^*}\to \tilde\D_{x}$ by $\tilde\beta_x\defeq\beta_{r_1}\beta_{r_1}\inv\cdots\beta_{r_n}\beta_{r_n}\inv\beta_r$ (as composition of partial maps). If $y=\epsilon_{s_1}\cdots\epsilon_{s_m}[s]$ is
another element of $\SG$ in normal form, we also define
$\tilde\omega(x,y)$ as the restriction of the unitary multiplier $\omega(r,s)\in\U\M(\D_r\cdot\D_{rs})$ to $\tilde\D_{xy}=\D_{r_1}\cdots\D_{r_n}\D_{rs_1}\cdots\D_{rs_m}\cdot\D_r\cdot\D_{rs}\sbe \D_r\cdot\D_{rs}$.

The following result is a generalization of Theorems~4.2 and 4.3 in \cite{SiebenTwistedActions}.

\begin{proposition}\label{prop:TwistedActionCorrespondence}
The pair $(\tilde\beta,\tilde\omega)$ (together with the domains $\tilde\D_x$) above defined is a (global) twisted action of $\SG$ on $B$.
Moreover, the assignment $(\beta,\omega)\mapsto (\tilde\beta,\tilde\omega)$ is a bijective correspondence between twisted partial actions of $\G$ on $B$ and twisted (global) actions of $\SG$ on $B$.
\end{proposition}
\begin{proof}
First, let us observe that if $x=\epsilon_{r_1}\cdots\epsilon_{r_n}[r]$ is an element of $\SG$ (not necessarily in normal form), then
$\tilde\D_x\defeq \D_{r_1}\cdots\D_{r_n}\cdot\D_r$. In fact, to check this it is enough to put $x$ in normal form and use
Proposition~\ref{prop:PropertiesTwistedPartialAction}\eqref{prop:PropertiesTwistedPartialAction:RelationBetweenDomains}. This justifies the fact that
$\tilde\D_{xy}=\D_{r_1}\cdots\D_{r_n}\D_{rs_1}\cdots\D_{rs_m}\cdot\D_r\cdot\D_{rs}$ for $x=\epsilon_{r_1}\cdots\epsilon_{r_n}[r]$ and $y=\epsilon_{s_1}\cdots\epsilon_{s_m}[s]$ because in this case
$xy=\epsilon_{r_1}\cdots\epsilon_{r_n}\epsilon_{rs_1}\cdots\epsilon_{rs_m}\epsilon_{r}[rs]$ (see Proposition~\ref{prop:FormulaProd}).
Also, since $x^*=[r^*]\epsilon_{r_1}\cdots\epsilon_{r_n}=\epsilon_{r^*r_1}\cdots\epsilon_{r^*r_n}[r^*]$ (see Proposition~\ref{PropFromInverse}(ii)), it follows that $\tilde\D_{x^*}=\D_{r^*}\cdot\D_{r^*r_1}\cdots\D_{r^*r_n}$.
Now, it is not difficult to see that the domain and range of $\beta_{r_1}\beta_{r_1}\inv\cdots\beta_{r_n}\beta_{r_n}\inv\beta_r$ are $\D_{r^*}\cdot\D_{r^*r_1}\cdots\D_{r^*r_n}$ and $\D_{r_1}\cdots\D_{r_n}\cdot\D_r$.
Therefore, $\tilde\beta_x\colon\tilde\D_{x^*}\to\tilde\D_{x}$ is a \Star{}isomorphism.
Also, observe that $\tilde\omega(x,y)$ is a unitary multiplier of $\tilde\D_{xy}$ since $\omega(r,s)$ is a unitary multiplier of $\D_r\cap\D_{rs}$ which contains $\tilde\D_{xy}$ as an ideal. It remains to check the axioms of
twisted action in \cite[Definition~4.1]{BussExel:Regular.Fell.Bundle} for the pair $(\tilde\beta,\tilde\omega)$. But since the $\tilde\beta$ and $\tilde\omega$ are just restrictions of $\beta$ and $\omega$, it is an exercise
to check that the axioms in~Definition~\ref{def:TwistedPartialAction} also hold for the pair $(\tilde\beta,\tilde\omega)$ and hence it is a twisted partial action of $\SG$ on $B$. To prove that $(\tilde\beta,\tilde\omega)$
is a (global) twisted action in the sense of \cite[Definition~4.1]{BussExel:Regular.Fell.Bundle}, by Proposition~\ref{prop:PropertiesTwistedPartialAction}(i), it is enough to check that $\tilde\D_x=\tilde\D_{xx^*}$ for all
$x\in\SG$. But if $x$ is written in normal form as $x=\epsilon_{r_1}\cdots\epsilon_{r_n}[r]$, then
$xx^*=\epsilon_{r_1}\cdots\epsilon_{r_n}\epsilon_r$ so that $\tilde\D_{xx^*}=\D_{r_1}\cdots\D_{r_n}\cdot\D_r=\tilde\D_{x}$.

Therefore, if $(\beta,\omega)$ is a twisted partial action of $\G$ on $B$, then the pair $(\tilde\beta,\tilde\omega)$ above defined is a (global) twisted action of $\SG$ on $B$. Conversely, if $(\tilde\beta,\tilde\omega)$ is
a twisted action of $\SG$ on $B$, then we can define $(\beta,\omega)$ by setting $\D_s\defeq \tilde\D_{[s]}$ (here, of course, $\tilde\D_x$ denotes the range of $\tilde\beta_x$ for all $x\in \SG$),
$\beta_s\defeq\tilde\beta_{[s]}$, and $\omega(s,t)\defeq\tilde\omega([s],[t])\in \U\M(\tilde\D_{[s][t]})$. Observe that by \cite[Lemma~4.6(v)]{BussExel:Regular.Fell.Bundle},
$\tilde\D_{[s][t]}=\tilde\D_{[s][s^*][st]}=\tilde\D_{[s]}\cdot \tilde\D_{[st]}=\D_s\cdot\D_{st}$, so that $\omega(s,t)$ is a unitary multiplier of $\D_s\cdot\D_{st}$. Again, it easy to check the axioms of twisted partial
action of the pair $(\beta,\omega)$. And it is clear that the assignments $(\beta,\omega)\mapsto (\tilde\beta,\tilde\omega)$ and $(\tilde\beta,\tilde\omega)\mapsto (\beta,\omega)$ are inverse to each other.
\end{proof}

\section{Fell bundles over $\G$ and $\SG$}

Recall that a \emph{Fell bundle} over an inverse semigroup $\G$ is a family $\A=\{\A_s\}_{s\in \G}$ of Banach spaces $\A_s$ together with
multiplication maps $\A_s\times\A_t\to \A_{st}$, involutions $\A_s\to \A_{s^*}$ for all $s,t\in \G$ and inclusions $\A_s\into\A_t$
whenever $s\leq t$ in $\G$. All this structure is required to be compatible in a suitable way. In addition, the norms on the fibers $\A_s$ are
required to satisfy the \cstar{}condition $\|a^*a\|=\|a\|^2$ for all $a\in \A_s$. Observe that $a^*a\in \A_{s^*s}$ whenever $a\in \A_s$.
In particular, the fiber $\A_e$ over an idempotent $e\in E(\G)$ is a \cstar{}algebra. One of the requirements in the definition of a Fell bundle is that $a^*a$ be a positive element of the \cstar{}algebra $\A_{s^*s}$ for all
$a\in \A_s$, that is, $a^*a=b^*b$ for some $b\in \A_{s^*s}$ (this is not automatically satisfied in general). There are two canonical \cstar{}algebras that can be constructed from a Fell bundle $\A$: one is the \emph{full
cross-sectional} \cstar{}algebra $C^*(\A)$ and the other is the \emph{reduced cross-sectional} \cstar{}algebra $C^*_\red(\A)$ of $\A$.
We refer the reader to \cite{Exel:noncomm.cartan} for the precise definition of Fell bundles and their cross-sectional \cstar{}algebras.

A famous result by Gelfand-Neumark asserts that every \cstar{}algebra is isomorphic to a concrete \Star{}algebra of operators, that is,
to some closed \Star{}subalgebra $A\sbe\Ls(\hils)$ of bounded operators on some Hilbert space $\hils$ (this is the reason why \cstar{}algebras are sometimes also called \emph{operator algebras}). An extension of this result
is also true for Fell bundles. A \emph{concrete} Fell bundle over $\G$
is a family $\A=\{\A_s\}$ of closed subspaces $\A_s\sbe \Ls(\hils)$, for some fixed Hilbert space $\hils$, such that
\begin{itemize}
\item $\A_s\A_t\sbe\A_{st}$ for all $s,t\in \G$;
\item $\A_s^*\sbe\A_{s^*}$ for all $s\in \G$; and
\item $\A_s\sbe\A_t$ whenever $s\leq t$ in $\G$.
\end{itemize}
Every (abstract) Fell bundle $\A$ is isomorphic to a concrete one. Indeed, this can be proved by applying the Gelfand-Neumark theorem mentioned above
to the full (or reduced) cross-sectional \cstar{}algebra $A=C^*(\A)$ and using the canonical representation of $\A$ into $A$ (see \cite{Exel:noncomm.cartan} for more details). Thus, given a Fell bundle $\A=\{\A_s\}_{s\in
\G}$, we may assume that there is a Hilbert space $\hils$ such that
$\A$ is a concrete Fell bundle in $\Ls(\hils)$ and such that $A=C^*(\A)$ is the \cstar{}subalgebra of $\Ls(\hils)$ defined as the closed linear span
of the fibers $\A_s\sbe\Ls(\hils)$. Observe that the fibers $\A_e$ over idempotents $e\in E=E(\G)$ are in this way \cstar{}subalgebras of $A\sbe\Ls(\hils)$. Let $\E=\{\A_e\}_{e\in \E}$ be the restriction of $\A$ to the
semilattice $E$.
By \cite[Proposition~4.3]{Exel:noncomm.cartan}, the \cstar{}algebra $B=C^*(\E)$ is isomorphic to the closed linear span of the fibers $\A_e$ with $e\in E$. So, we may further assume that $B$ is a \cstar{}subalgebra of
$A\sbe\Ls(\hils)$.
Hence, there is no loss of generality in assuming that a given Fell bundle is concrete and we shall usually do so in what follows.

\begin{definition}
Given a \cstar{}algebra $A$, we write $\Pl(A)$ for the semigroup of all closed subspaces of $A$ with respect to the
canonical multiplication:
$$M\cdot N\defeq \cspn MN\quad\mbox{for all }M,N\in \Pl(A).$$
If $\hils$ is a Hilbert space, we write $\Pl(\hils)$ for $\Pl(\Ls(\hils))$.
\end{definition}

Observe that $\Pl(A)$ is not an inverse semigroup in general, even if we restrict attention to subspaces $M\sbe A$ that are ternary ring of operators
in the sense that $M\cdot M^*\cdot M=M$. The problem is that $\Pl(A)$ is too big and idempotents do not commute in general. One way to find an inverse
subsemigroup inside $\Pl(A)$ is to consider a fixed \cstar{}subalgebra $B\sbe A$ and consider the subset $\Pl(A,B)\sbe\Pl(A)$ consisting of
all closed subspaces $M\sbe A$ satisfying
$$MM^*M\sbe M,\quad M^*M\sbe B,\quad MM^*\sbe B,\quad MB\sbe M\quad \mbox{and}\quad BM\sbe M.$$
Then it is not difficult to prove that $\Pl(A,B)$ is in fact an inverse subsemigroup of $\Pl(A)$.
This is done in the proof of \cite[Proposition~8.6]{BussExel:Regular.Fell.Bundle}.
The inverse of $M$ in $\Pl(A,B)$ is the adjoint $M^*=\{m^*\st m\in M\}$. If $A$ is $\Ls(\hils)$ and $B$ is a \cstar{}subalgebra of $\Ls(\hils)$, we write $\Pl(\hils,B)$ for $\Pl(\Ls(\hils),B)$.

In particular, if $\A=\{\A_s\}_{s\in \G}$ is a concrete Fell bundle in $\Ls(\hils)$ for some Hilbert space $\hils$,
we may consider the \cstar{}algebra $B=C^*(\E)\sbe\Ls(\hils)$ defined as the closed linear span of the fibers $\A_e$ with $e\in E(S)$.
Then $\Pl(\hils,B)$ is an inverse semigroup containing all the fibers $\A_s\sbe \Ls(\hils)$.
This assertion includes in particular the fact that each $\A_s$ is a ternary ring of operators (this is proved in
\cite[?]{BussExel:Regular.Fell.Bundle}). Observe that the fibers generate an inverse subsemigroup of $\Pl(\hils,B)$,
namely the inverse semigroup consisting of all finite products (that is, closed linear spans)
of the form $\A_{s_1}\cdot\A_{s_2}\cdots \A_{s_n}$ with $s_i\in \G$. If $\A$ is saturated,
meaning that $\cspn\A_s\A_t=\A_{st}$ for all $s,t\in \G$, then in a sense the Fell bundle $\A$ itself may be "confused" with this inverse semigroup.

\begin{theorem}\label{theo:FellBundlesoverGandS(G)}
Let $\A=\{\A_s\}_{s\in \G}$ be a concrete Fell bundle in $\Ls(\hils)$ over an inverse semigroup $\G$.
Then the map $\pi\colon\G\to \Pl(\hils)$ given by $\pi(s)\defeq\A_s$ is a partial homomorphism.
Moreover, $\pi$ is a homomorphism if and only if $\A$ is saturated. Let $\Sh\pi\colon\SG\to \Pl(\hils)$ be the representation associated to $\pi$
as in Proposition~\ref{prop:UniProp}, and let $\Sh\A_{x}\defeq\Sh\pi(x)$ for all $x\in \SG$. Then $\Sh\A=\{\Sh\A_{x}\}_{x\in \SG}$ is a saturated Fell bundle over $\SG$. Given $s_1,\ldots,s_n,s\in \G$, the fiber $\Sh\A_x$
over $x=\e{s_1}\cdots\e{s_n}[s]$ is given by
\begin{equation}\label{eq:FibersOfCheckA}
\Sh\A_x=\A_{s_1}\cdot\A_{s_1^*}\cdots\A_{s_n}\cdot\A_{s_n^*}\cdot\A_s.
\end{equation}
The assignment $\A\mapsto\tilde\A$ is a bijective correspondence between Fell bundles over $\G$ and saturated Fell bundles over $\SG$.
Furthermore, the canonical Fell bundle morphism $\Sh\A\to \A$ induces \cstar{}algebra isomorphisms $C^*(\Sh\A)\cong C^*(\A)$, $C^*_{\red}(\Sh\A)\cong C_\red^*(\A)$ and $C^*(\Sh\E)\cong C^*(\E)$, where $\Sh\E$ and $\E$ are
the restrictions of $\Sh\A$ and $\A$ to $E(\SG)$ and $E(\G)$, respectively.
\end{theorem}
\begin{proof}
To prove that $\pi$ is a partial homomorphism, we have to show that
$$\A_s\cdot\A_t\cdot \A_{t^*}=\A_{st}\cdot\A_{t^*}\quad\mbox{and}\quad \A_{s^*}\cdot\A_s\cdot\A_t=\A_{s^*}\cdot\A_{st}\quad \mbox{for all }s,t\in \G.$$ Since $\A_s\A_t\sbe\A_{st}$, we have $\A_s\cdot\A_t\cdot
\A_{t^*}\sbe\A_{st}\cdot\A_{t^*}$. On the other hand, observe that $\A_{se}\sbe\A_s$ for all $e\in E(\G)$ because $se\leq s$ in $\G$. Applying this to $e=tt^*$ and using that $\A_{t^*}=\A_{t^*}\cdot\A_{t}\cdot\A_{t^*}$, we
get
$$\A_{st}\cdot \A_{t^*}=\A_{st}\cdot\A_{t^*}\cdot\A_{t}\cdot\A_{t^*}\sbe \A_{stt^*}\cdot \A_t\cdot\A_{t^*}\sbe\A_{s}\cdot\A_t\cdot\A_{t^*}.$$
Similarly, $\A_{s^*}\cdot\A_s\cdot\A_t=\A_{s^*}\cdot\A_{st}$. Thus, $\pi$ is a partial homomorphism.

Of course, $\pi$ is a homomorphism
if and only if $\A_{st}=\pi(st)=\pi(s)\pi(t)=\A_s\cdot\A_t$, that is, $\A$ is saturated.
Observe that the image of $\pi$ (and hence of $\Sh\pi$) is contained in $\Pl(\hils,B)$, where $B=C^*(\E)$ (see comments above).
Since $\Sh\pi$ is a representation, we have $\Sh\pi(xy)=\Sh\pi(x)\Sh\pi(y)$ and $\Sh\pi(x^*)=\Sh\pi(x)^*$, that is, we have
$\Sh\A_{xy}=\Sh\A_x\cdot\Sh\A_y$ and $\Sh\A_{x^*}=\Sh\A_x^*$ for all $x,y\in \SG$.
If $x\leq y$ in $\SG$, then $\Sh\pi(x)\leq \Sh\pi(y)$ in $\Pl(\hils,B)$, that is, $\Sh\A_x\sbe\Sh\A_y$
whenever $x\leq y$ (the natural order of $\Pl(\hils,B)$ is given by inclusion of subspaces).
This implies that $\Sh\A$ is a (concrete) saturated Fell bundle over $\SG$.
If $x=\e{s_1}\cdots\e{s_n}[s]$, observe that
$$\Sh\pi(x)=\Sh\pi(\e{s_1})\cdots\Sh\pi(\e {s_n})\Sh\pi(\e s)=\pi(s_1)\pi(s_1^*)\cdots\pi(s_n)\pi(s_n^*)\pi(s).$$
This is exactly Equation~\eqref{eq:FibersOfCheckA}.

If $\tilde\A$ is a saturated Fell bundle over $\SG$, then we may "restrict" $\tilde\A$ to $\G$ (using the canonical inclusion map $\iota_\G\colon\G\to\SG$) and get a Fell bundle $\A$ over $\G$. The two constructions
$\A\mapsto\tilde\A$ and $\tilde\A\mapsto \A$ are easily seen to be inverse to each other, so that Fell bundles over $\G$ correspond bijectively to saturated Fell bundles over $\SG$.

The final assertion will follow from \cite[Theorem~8.4]{BussExel:Regular.Fell.Bundle}, after we check that the canonical morphism
$$(\iota,\partial)\colon(\Sh\A,\SG)\to (\A,\G)$$
gives $\Sh\A$ as a \emph{refinement} of $\A$ in the sense of \cite[Definition~8.1]{BussExel:Regular.Fell.Bundle},
where $\partial\colon \SG\to \G$ is the degree map and $\iota\colon\Sh\A\to\A$ is given by the inclusions $\Sh\A_x\into\A_{\partial(x)}$
(observe that $\Sh\A_x\sbe\A_{\partial(x)}$ by Equation~\eqref{eq:FibersOfCheckA}).
First, we already know from Proposition~\ref{prop:NormalFormDegree} that $\partial$ is (surjective and) essentially injective.
And, by definition, $\iota\colon\Sh\A_x\to \A_{\partial(x)}$ is the inclusion map, so it is injective. Finally, given $s\in \G$, the condition
$$\A_s=\overline{\sum\limits_{\partial(x)=s}\iota(\Sh\A_x)}$$
appearing in \cite[Definition~8.1]{BussExel:Regular.Fell.Bundle} is obviously satisfied since $\iota(\Sh\A_{[s]})=\A_s$.
\end{proof}

\begin{remark}\label{rem:refinementsAndRegularity}
The above result improves \cite[Proposition~8.6]{BussExel:Regular.Fell.Bundle}.
To be more precise, in \cite[Proposition~8.6]{BussExel:Regular.Fell.Bundle} we proved that every Fell bundle
$\A=\{\A_s\}_{s\in \G}$ has a saturated refinement $\B=\{\B_t\}_{t\in T}$, where $T$ is another inverse semigroup which, \emph{a priori}, depends on
the Fell bundle $\A$ (see proof of Proposition~8.6 in \cite{BussExel:Regular.Fell.Bundle}). Above we have proved that $T$ can be chosen to be $\SG$, a more natural choice which depends only on $\G$. So, for instance, if we
know that $\G$ is countable, then so is $T=\SG$ (this was not clear in \cite[Proposition~8.6]{BussExel:Regular.Fell.Bundle}). Also, if the fibers $\A_s$ are separable, then so are the fibers $\B_t$ because each one is a
subspace of some $\A_s$. The refinement obtained in the above theorem also preserves (local) regularity (see \cite{BussExel:Regular.Fell.Bundle} for the definition of (local) regularity of Fell bundles): in fact, observe
that each fiber $\Sh\A_x$ is an ideal of $\A_{\partial(x)}$ (in the sense of \cite[Definition~6.1]{Exel:TwistedPartialActions}). So, the result follows from \cite[Proposition~6.3]{Exel:TwistedPartialActions}.
\end{remark}

\section{Fell bundles and partial twisted actions}

Given a twisted partial action $(B,\beta,\omega)$ of $\G$ on a \cstar{}algebra $B$ as in Definition~\ref{def:TwistedPartialAction}, there is a canonical associated Fell bundle $\A$ over $\G$ as described in
\cite{BussExel:Regular.Fell.Bundle} for (global) twisted actions (where we get saturated Fell bundles). The Fell bundle $\A$ is defined as follows:

The fiber $\A_s$ over $s\in \G$ is defined by $\A_s\defeq \{(a,s): b\in \D_s\}=\D_s\delta_s$, where we write $a\delta_s$ for the element $(a,s)$. The operations on $\A$ are defined by
\begin{equation*}\label{eq:DefProductFellBundleFromTwistedAction}
(a\delta_s)\cdot(b\delta_t)\defeq\beta_s\big(\beta_s^{-1}(a)b\big)\omega(s,t)\delta_{st}
\end{equation*}
and
\begin{equation*}\label{eq:DefInvolutionFellBundleFromTwistedAction}
(a\delta_s)^*\defeq \beta_s^{-1}(a^*)\omega(s^*,s)^*\delta_{s^*}
\end{equation*}
and the inclusion maps $j_{t,s}\colon\A_s\into \A_t$ for $s\leq t$ in $\G$ are defined by
\begin{equation*}\label{eq:DefInclusionsFellBundleFromTwistedAction}
j_{t,s}\colon\B_s\to \B_t\quad\mbox{by}\quad j_{t,s}(a\delta_s)\defeq a\omega(t,s^*s)^*\delta_t.
\end{equation*}
The above operations are exactly the same as the ones defined in \cite{BussExel:Regular.Fell.Bundle}. And the proof that $\A$ is in fact a Fell bundle can also be made in the same way. The only difference here is that $\A$
is not necessarily saturated because of the partiality of our twisted action $(\beta,\omega)$. In fact, it is easy to see that $\A$ is saturated if and only if $(\beta,\omega)$ is a (global) twisted action (as defined in
\cite{BussExel:Regular.Fell.Bundle}). In any case, the Fell bundle $\A$ is always regular . This means that the fibers $\A_s$ are regular as imprimitivity Hilbert $I_s$-$J_s$-bimodules, where $I_s\defeq \A_s\cdot\A_s^*$ and
$J_s=\A_s^*\cdot\A_s$ (closed linear spans), that is (see \cite{BussExel:Regular.Fell.Bundle} for the precise definition), there is a family $u=(u_s)_{s\in \G}$ of unitary multipliers $u_s\colon J_s\to\A_s$ for the
imprimitivity bimodules $\A_s$ (recall that a multiplier of an imprimitivity Hilbert $I,J$-bimodule $\F$ is, by definition, an adjointable operator $v\colon J\to \F$). For the Fell bundle $\A$ above defined we may take $u_s$
as the multiplier defined by $u_s(x\delta_{s^*s})\defeq \beta_s(x)\delta_s$ for all $x\in \D_{s^*}$. It is easy to see that $I_s=\A_s^*\A_s=\D_{s^*}\delta_s$ so that $u_s$ is a well-defined map from $J_s$ to $\A_s$.
Moreover, $u_s$ is adjointable with adjoint given by $u_s^*(y\delta_s)=\beta_s^{-1}(y)\delta_{s^*s}$ for all $y\in \D_s$. Observe that $u_s$ is in fact a unitary multiplier, that is, $u_s^*u_s=1_{s^*}$ and $u_su_s^*=1_s$,
where $1_s$ is the unit multiplier of $J_s\cong \D_s$ (isomorphism of \cstar{}algebras). Also, observe that $u_e=1_e$ for all $e\in E(\G)$.

Hence, to a twisted partial action $(\beta,\omega)$ we attach a pair $(\A,u)$ consisting of a (regular) Fell bundle $\A$ over $\G$ and a family $u=(u_s)_{s\in \G}$ of unitary multipliers $u_s$ for $\A_s$ such that $u_e=1_e$
for all $e\in E(\G)$.

Conversely, to such a pair $(\A,u)$ it is possible to construct a twisted partial action $(\beta,\omega)$ using ideas as in \cite[Section~3]{BussExel:Regular.Fell.Bundle} as follows: given a regular concrete Fell bundle
$\A=\{\A_s\}_{s\in \G}$ in $\Ls(\hils)$, a family of unitary multipliers $u=\{u_s\}_{s\in \G}$ as above may be viewed concretely as partial isometries $u_s\in \Ls(\hils)$ satisfying:
\begin{equation}\label{eq:partial.Isometries.Regularity}
u_s\A_s^*=\A_s\cdot \A_s^*,\quad \A_s^*u_s=\A_s^*\cdot\A_s,\quad u_su_s^*=1_s,\quad u_s^*u_s=1_{s^*},\mbox{ and}\quad u_e=1_e
\end{equation}
for all $s\in \G$ and $e\in E(\G)$, where $1_s$ denotes the unit multiplier of $\A_s\cdot\A_s^*$, that is, the (orthogonal) projection onto the subspace $\A_s\cdot\A_s^*\cdot \hils\sbe \hils$. The associated twisted partial
action $(\beta,\omega)$ is then given by
$$\beta_s\colon\D_{s^*}\to\D_s,\quad\beta_s(x)\defeq u_sxu_s^*,\quad\mbox{and}\quad\omega(s,t)\defeq u_su_tu_{st}^*,$$
where $\D_s\defeq\A_s\cdot\A_s^*$ (closed linear span). The same ideas as in \cite{BussExel:Regular.Fell.Bundle} may be used to prove that the pair $(\beta,\omega)$ above defined is in fact a twisted partial action and the
following result (see \cite[Corollary~4.16]{BussExel:Regular.Fell.Bundle} for the saturated case) may be obtained:

\begin{proposition}
The above construction $(\beta,\omega)\mapsto (\A,u)$ is a bijective correspondence between (isomorphism classes of) twisted partial actions $(\beta,\omega)$ of $\G$ and regular Fell bundles $(\A,u)$ over $G$.
\end{proposition}

The \cstar{}algebra $B$ above where $\G$ acts through $(\beta,\omega)$ is, via the above correspondence, canonically isomorphic to the cross-sectional \cstar{}algebra $C^*(\E)$ of the restriction $\E$ of $\A$ to the
idempotents $E=E(\G)$. In fact, observe that by the construction of $\A$, each fiber $\A_e= \D_e\delta_e\cong \D_e$ over an idempotent $e\in E$ embeds into $B$. Since the ideals $\D_s$ for $s\in \G$ span a dense subspace of
$B$, the same is true for $\D_e$ with $e\in E$ because for every $s\in \G$ we have $\D_s\sbe\D_{ss^*}$ (see Proposition~\ref{prop:PropertiesTwistedPartialAction}(i)). Hence, the embeddings $\A_e\into B$ extend to an
isomorphism $C^*(\E)\congto B$ by \cite[Proposition~4.3]{Exel:noncomm.cartan}.

Moreover, the correspondence $(B,\G,\beta,\omega)\mapsto \A$ is also compatible with crossed products in the sense that
\begin{equation}\label{eq:IsoCrossedProducts}
B\rtimes_{\beta,\omega}\G\cong C^*(\A)\quad\mbox{and}\quad B\rtimes_{\beta,\omega}^\red\G\cong C^*_\red(\A).
\end{equation}
We did not define the above full $B\rtimes_{\beta,\omega}\G$ and reduced $B\rtimes_{\beta,\omega}^\red\G$ crossed products for a twisted partial action $(\beta,\omega)$ of $\G$ on $B$, but they can be defined (and the above
isomorphisms can be obtained) in the same way as explained in \cite[Section~6]{BussExel:Regular.Fell.Bundle} for twisted (global) actions.

Now, if $\A$ is a regular Fell bundle concrete in $\Ls(\hils)$, then, as already mentioned in Remark~\ref{rem:refinementsAndRegularity}, the corresponding Fell bundle $\tilde\A$ over $\SG$ as in
Theorem~\ref{theo:FellBundlesoverGandS(G)} is also regular (and concrete in $\Ls(\hils)$). Moreover, if $u=\{u_s\}_{s\in \G}$ is a family of partial isometries on $\hils$ satisfying~\eqref{eq:partial.Isometries.Regularity}
for $\A$, then the family $\tilde u=\{\tilde u_x\}_{x\in \SG}$ defined by
$$\tilde u_x\defeq u_{s_1}u_{s_1}^*\cdots u_{s_n}u_{s_n}^*u_s\quad\mbox{whenever }x=\epsilon_{s_1}\cdots\epsilon_{s_n}[s]\quad\mbox{is in normal form},$$
also satisfies~\eqref{eq:partial.Isometries.Regularity} for $\tilde\A$. Moreover, it is not difficult to see that if $(\beta,\omega)$ is the twisted partial action corresponding to $(\A,u)$, then the twisted (global) action
$(\tilde\beta,\tilde\omega)$ associated to $(\tilde\A,\tilde u)$ as above is the same as the one already obtained in Propostion~\ref{prop:TwistedActionCorrespondence}. In particular,
Theorem~\ref{theo:FellBundlesoverGandS(G)} and~\eqref{eq:IsoCrossedProducts} yields the following consequence:

\begin{corollary}
Let $(\beta,\omega)$ be a twisted partial action of $\G$ on $B$ and let $(\tilde\beta,\tilde\omega)$ the corresponding twisted action of $\SG$ on $B$. Then
$$B\rtimes_{\beta,\omega}\G\cong B\rtimes_{\tilde\beta,\tilde\omega}\SG\quad\mbox{and}\quad
    B\rtimes_{\beta,\omega}^\red\G\cong B\rtimes_{\tilde\beta,\tilde\omega}^\red\SG.$$
\end{corollary}

Of course, this also covers the case of (untwisted) partial actions of inverse semigroups on \cstar{}algebras (and, in particular, on locally compact spaces):

\begin{corollary}
Let $\beta$ be a partial action of $\G$ on $B$ and let $\tilde\beta$ the corresponding action of $\SG$ on $B$. Then
$$B\rtimes_\beta\G\cong B\rtimes_{\tilde\beta}\SG\quad\mbox{and}\quad
    B\rtimes_{\beta}^\red\G\cong B\rtimes_{\tilde\beta}^\red\SG.$$
\end{corollary}

%%%%%%%%%%%%%%%%%%%%%%%%%%%%%%%%%%%%%%%%%%%%%%%%% References %%%%%%%%%%%%%%%%%%%%%%%%%%%%%%%%%%%

\vskip 1pc
\end{document}